 \documentclass[preprint,review,12pt]{elsarticle}

\usepackage{CJK,fancybox,color,graphicx,amsfonts,amscd,amsmath,amsthm,amssymb,enumerate, subfigure}
\usepackage{latexsym,epsfig,changebar,indentfirst}

\biboptions{numbers,sort&compress}  %ѹËõ²Î¿¼ÎÄÏ×£¬¿ÉÒÔ³öÏÖ 1--20

\usepackage{pict2e}

\newtheorem{thm}{Theorem}                % ¶¨Àí£¬ Èç¹û²»²ÉÓÃÕ½ںÅ×öǰ׺£¬Ôò²»ÓÃ[section]

\newtheorem{proposition}[thm]{Proposition}  % ÃüÌ⣬ Èç¹û²»²ÉÓÃÕ½ںÅ×öǰ׺£¬Ôò²»ÓÃ[section]
\renewcommand{\figurename}{Fig.}            % ½«Figure  ת»¯Îª  Fig.

\begin{document}

\begin{frontmatter}



\title{  Well-balanced finite difference WENO schemes\\ for the blood flow model}

%\tnotetext[label1]{The research of the first author is supported
%by the National Natural Science Foundation of  China (No. 41171183, 11201254) and the Project for Scientific Plan of Higher Education in Shandong Providence of China %(No. J12LI08).  The second author is sponsored by MIUR, Prin2009, Metodi numerici innovativi per problemi iperbolici con applicazioni in fluidodinamica, teoria cinetica %e biologia computazionale. }

\author[label2]{Zhenzhen Wang}
\ead{wzz0667@163.com}
\author[label2]{Gang Li\corref{cor1}}
\ead{gangli1978@163.com}
\author[label3]{Olivier Delestre}
\ead{Delestre@unice.fr}

\address[label2]{ School of Mathematics and Statistics, Qingdao University, Qingdao, Shandong 266071, P.R. China   }
\address[label3]{ Laboratory J.A. Dieudonn\'{e} $\&$ EPU Nice Sophia, UMR 7351 Parc Valrose, 28 Avenue Valrose 06108 Nice Cedex 02, 06000 Nice, France }

\cortext[cor1]{Corresponding author. Tel.: +86-0532-85953660. Fax: +86-0532-85953660.}

\begin{abstract}
The blood flow model maintains the steady state solutions, in which the flux gradients are non-zero but
exactly balanced by the source term.
In this paper, we design high order finite difference weighted non-oscillatory (WENO) schemes to this model with such well-balanced property and at the same time keeping genuine high order accuracy.
Rigorous theoretical analysis as well as extensive numerical results all indicate that the resulting schemes verify high order accuracy, maintain the well-balanced property, and keep good resolution for smooth and discontinuous solutions.
\end{abstract}


\end{frontmatter}


%\linenumbers


\section{Introduction}     \label{introduction}
\setcounter{equation}{0}
\setcounter{figure}{0}
\setcounter{table}{0}

In this paper, we are interested in numerical simulation for the blood flow model by high order finite difference schemes. The numerical simulations with high order accuracy have a wide applications in medical engineering  \cite{Formaggia2006,Kolachalama2007}.
As quoted by Xiu and Sherwin \cite{Xiu2007}, the blood flow in arteries model was written long time ago by Leonhard Euler in 1775. However, the model is too difficult to solve \cite{Parker2009}. Herein, for the sake of simplicity, we neglect the friction term and consider the following governing equations
\begin{equation}   \label{eqn-1}
\left\{
\begin{array}{l}
\partial_{t}A+\partial_{x}Q=0,\\
\partial_{t}Q+\partial_{x}\left(\frac{Q^{2}}{A}+\frac{k}{3\rho \sqrt{\pi} } A^{\frac{3}{2}}  \right) =
\frac{kA}{\rho \sqrt{\pi}}\partial_{x} ( \sqrt{ A_{0} } ),
\end{array}
\right.
\end{equation}
where $A$ is the cross-sectional area ($A = \pi R^{2}$ with $R$ being the radius of the vessel),
$Q = Au$ denotes the discharge, $u$ means the flow velocity, and $\rho$ stands for the blood density. $k$ represents the stiffness arterial.
 In addition, $A_0$ is the cross section at rest (\emph{i.e.,} $A_0 = \pi R_0^2$ with $R_0$ being the radius of the vessel, which may be variable in the case of aneurism, stenosis or taper).

The blood flow model (\ref{eqn-1}) with the source term are also called as \textit{balance laws}. This model
can admit the following steady state solutions, also called ``man at eternal rest'' (by analogy to the ``lake at rest''
 in the shallow water equations)
\begin{equation}    \label{steady-1D}
u=0, \; \;  \; A = A_0.
\end{equation}
For the steady state solutions (\ref{steady-1D}), the source term is exactly balanced by the non-zero flux gradient.
Thus it is desirable to maintain the balance between the flux gradient and the source term at the discrete level. But such a balance is usually neither a constant nor a polynomial function.
%So, it is difficult for the standard numerical schemes to preserve all the solutions of (\ref{steady-1D-1}) and (\ref{steady-1D-2}).
So standard numerical schemes usually fail to capture the steady state solutions well and  may generally introduce spurious oscillations. The mesh must be extremely refined to reduce the size of these oscillations, but this strategy is impractical for multi-dimensional cases due to the high computational costs.
Berm\'{u}dez and V\'{a}zquez \cite{Bermudez} in 1994 proposed the idea of ``exact conservation property'', which means that a scheme is exactly compatible with the steady state solutions. This property is also known as ``well-balanced'' property and is crucial for the balance between the flux gradient and the source term. An efficient scheme should satisfies this well-balanced property. Such schemes are often regarded as well-balanced schemes after the pioneering works of Greenberg et al. \cite{Greenberg1996,Greenberg1997}.
The well-balanced schemes can preserve exactly these steady state solutions up to the machine error free of excessive mesh refinement and save computational cost accordingly. Moreover, the important advantage of well-balanced schemes over non-well-balanced schemes is that they can accurately resolve small perturbations of such steady state solutions with relatively coarse meshes \cite{Noelle2010, Xing2011-1}. More information about well-balanced schemes can be found in the lecture note \cite{Noelle2010}. Many researchers have developed well-balanced schemes for the shallow water
equations admitting the still water steady state using different approaches, see, \emph{ e.g.,} \cite{LeVeque1998, Perthame2001, Xu2002, Audusse2004, Xing2005, Xing2014} and the references therein. It is a challenging to design well-balanced schemes for the moving water equilibrium of the shallow water equations.
Most well-balanced schemes for the still water steady state cannot preserve the moving water equilibrium automatically.  A few attempts can be found in \cite{Noelle2007, Xing2014-2, Bouchut2010}. In addition, the research of the well-balanced schemes for the Euler equations under gravitational fields is also an active subject
\cite{Xu2007,Xu2010,Xu2011,Kappeli2014,Xing2013,Li2015}.

In recent years, there have been many interesting attempts proposed in the literature to derive well-balanced schemes for the blood flow model. For example,  Delestre \emph{et al.} \cite{Delestre2013} present a well-balanced finite volume scheme for the blood flow model
based on the conservative governing equations \cite{Wibmer2004,Cavallini2008,Cavallini2010}. Recently, M\"{u}ller \emph{et al.} \cite{Muller2013} constructed a well-balanced high order finite volume for the blood flow in elastic vessels with varying mechanical properties.
More recently, Murillo \emph{et al.} \cite{Murillo2015} present an energy-balanced approximate solver for the blood flow model with upwind discretization for the source term.

The main objective of this paper is to design a well-balanced finite difference weighted non-oscillatory (WENO) scheme which maintains the well-balanced property and at the same time keeps genuinely high order accuracy for the general solutions of the blood flow model, based on a special splitting of the source term into two parts which are discretized separately.

This paper is organized as follows: in Section \ref{the-scheme}, we propose a high order well-balanced finite difference WENO scheme. Extensive numerical experiments are carried out in Section \ref{the-results}. Conclusions are given in Section \ref{the-conclusion}.

\section{Well-balanced WENO schemes   } \label{the-scheme}

In this section, we present high order well-balanced WENO schemes for the blood flow model satisfying the steady state solution
(\ref{steady-1D}).

\subsection{ Notations }
For simplicity, we assume that the grid points $\{x_j \}$ are
uniformly distributed  with cell size $\Delta x = x_{j+1}-x_j$ and we denote the cells by $I_j=\left[ x_{j-1/2}, \; x_{j+1/2}\right]$
 with $x_{j+1/2}=x_j+\Delta x/2$ as the center of the cell $I_j$.

\subsection{A review of finite difference WENO schemes}

The first finite difference WENO scheme was designed in 1996 by Jiang and Shu \cite{Jiang1996} for hyperbolic conservation laws. More detailed information of WENO schemes can be found in the lecture note \cite{Shu1997}. For the latest advances regarding WENO schemes, we refer to the review \cite{Shu2009}. We begin with the description for the 1D scalar conservation laws
\begin{equation}    \label{laws-1D}
u_t+f(u)_x=0.
\end{equation}
High order semi-discrete conservative finite difference schemes of ($\ref{laws-1D}$)
can be formulated as follows
\begin{equation}  \label{semi-scheme}
\frac{{\rm d}}{{\rm d}t} u_j(t)   =   -\frac{1}{\Delta
x}\left(\hat{f}_{j+1/2}-\hat{f}_{j-1/2}\right),
\end{equation}
where $u_j(t)$ is the numerical approximation to the point value
$u(x_j,t)$, and the numerical flux $\hat{f}_{j+1/2}$ is used to approximate
$h_{j+1/2}=h\left(x_{j+1/2}\right)$ with high order accuracy. Here $h(x)$
is implicitly defined as in \cite{Jiang1996}
$$
f(u(x))=\frac{1}{\Delta x}\int^{x+\Delta x/2}_{x-\Delta
x/2}h(\xi)d\xi.
$$

We take upwinding into account to maintain the numerical stability and splitting a general flux into two
parts
$$
f(u) = f^+(u) + f^-(u),
$$
where $ \displaystyle  \frac{ {\rm d} f^+(u) }{{\rm d} u} \geq 0$ and $ \displaystyle  \frac{ {\rm d} f^-(u) }{{\rm d} u} \leq 0$.
One example is the simple Lax-Friedrichs flux
\begin{equation}
f^{\pm}(u) = \frac{1}{2}(f(u) \pm \alpha u),
\end{equation}
where $ \alpha= \max\limits_u \big|\lambda(u)\big|$ with $\lambda(u)$ being the eigenvalues of the Jacobian $f'(u)$, and the maximum
is taken over the whole region. With respect to $f^+(u)$ and $f^-(u)$, we can get numerical fluxes $\hat{f}^+_{j+1/2}$ and $\hat{f}^-_{j+1/2}$ using the WENO reconstruction, respectively. Finally, we get the numerical fluxes as follows $$\hat{f}_{j+1/2} = \hat{f}^{+}_{j+1/2} + \hat{f}^{-}_{j+1/2}.$$

By means of the WENO approximation procedure, $\hat{f}^+_{j+1/2}$ is expressed as \cite{Jiang1996}
\begin{equation}   \label{weno}
\hat{f}^+_{j+1/2} = \sum^{r}_{k=0}\omega_kq^r_k\left(f^+_{j+k-r},\ldots,f^+_{j+k}\right),
\end{equation}
where $\omega_k$ is the nonlinear weight,
$f^+_i=f^+(u_i),\,i=j-r,\ldots,j+r,$
and
\begin{equation}
q^r_k\left(\mbox{g}_0,\ldots,\mbox{g}_{r}\right)=\sum^{r}_{l=0}a^r_{k,l}\mbox{g}_l
\end{equation}
 is the low order approximation to $\hat{f}^+_{j+1/2}$ on the $k$th
stencil $S_k=(x_{j+k-r},\ldots,x_{j+k}),k=0,1,\ldots,r$, and
$a^r_{k,l}, \; 0 \leq k , \, l \leq r$ are constant coefficients, see
\cite{Shu1997} for more details.

The nonlinear weights $\omega_k$ in (\ref{weno}) satisfy
$$ \sum^{r}_{k=0}\omega_k=1, $$
and are designed to yield $(2r+1)$th-order accuracy in smooth regions
of the solution. In \cite{Jiang1996,Shu1997}, the nonlinear weight
$\omega_k$ is formulated as
\begin{equation}    \label{linearweight}
\omega_k=\frac{\alpha_k}{
\sum\limits^{r}_{l=0}\alpha_l},  \; \; \mbox{with} \; \;
\alpha_k=\frac{C^r_k}{\left(\varepsilon_{_{\text{WENO}}} + IS_k\right)^2}, \; \; k=0,1,\ldots,r,
\end{equation}
where $C^r_k$ is the linear weight. $IS_k$ is a smoothness indicator of $f^+(u)$ on
stencil $S_k,k=0,1,\ldots,r$, and $\varepsilon_{_{\text{WENO}}} $ is a small constant
used here to avoid the denominator becoming zero, $ \varepsilon_{_{\text{WENO}}} = 10^{-6}$ is used in all test cases in this paper. We employed
the smoothness indicators proposed in \cite{Jiang1996,Shu1997}, \emph{i.e.,}
$$
IS_k
=
\sum^{r}_{l=1}\int^{x_{j+1/2}}_{x_{j-1/2}}(\Delta x)^{2l-1}\left(q_k^{(l)}\right)^2dx,$$
where $q^{(l)}_k$ is the $l$th-derivative of $q_k(x)$ which is the reconstruction
polynomial of $f^+(u)$ on stencil $S_k$ such that£º
$$
\frac1{\Delta x} \int_{I_i} q_k(x)dx=f^+_i,\,i=j+k-r,\ldots,j+k.
$$

The WENO approximation procedure for $\hat{f}^-_{j+\frac{1}{2}}$
is a mirror symmetry to that of $\hat{f}^+_{j+1/2}$ with respect to
$x_{j+1/2}$.

Consequently, the numerical flux $\hat{f}_{j+1/2}$ is then calculated by
$$\hat{f}_{j+1/2}=\hat{f}^+_{j+1/2}+\hat{f}^-_{j+1/2}.$$
Ultimately, we obtain the semi-discrete scheme (\ref{semi-scheme}).

\subsection{Well-balanced WENO schemes for the blood flow model } \label{scheme-1D}

In order to design well-balanced schemes, we firstly split the source term
 $ \frac{kA}{\rho \sqrt{\pi} }\partial_{x} ( \sqrt{ A_{0}}) $ into two terms
$ \frac{k}{ \rho \sqrt{\pi}} (A - A_0) \partial_{x} ( \sqrt{ A_{0}} )  + \frac{k}{3 \rho \sqrt{\pi}} \partial_{x} \left(A_0^{\frac{3}{2}}\right)  $ in a equivalent form.  Therefore the original system (\ref{eqn-1}) becomes
\begin{equation}   \label{eqn-1-modified}
\left\{
\begin{array}{l}
\partial_{t}A + \partial_{x}Q=0,\\
\partial_{t}Q + \partial_{x}\left( \frac{Q^{2}}{A}+\frac{k}{3\rho \sqrt{\pi} } A^{\frac{3}{2}}  \right) =
 \frac{k}{\rho \sqrt{\pi} } (A - A_0) \partial_{x} ( \sqrt{A_{0} } ) + \frac{k}{3 \rho \sqrt{\pi}} \partial_{x} \left(A_0^{\frac{3}{2}}\right),
\end{array}
\right.
\end{equation}
which can be denoted in a compact vector form
$$
U_t + f(U)_x = S_1 + S_2,
$$
where $U= (A, \; Q)^{T}$, $f(U) = \left(Q, \;  \frac{Q^{2}}{A}+\frac{k}{3\rho \sqrt{\pi} } A^{\frac{3}{2}}   \right)$,
 $S_1 = \left(0, \;    \frac{k}{ \rho \sqrt{\pi}} (A - A_0) \partial_{x}  (\sqrt{A_{0}})    \right)^T$ and
$S_2 = \left( 0, \;  \frac{k}{3\rho \sqrt{\pi} }  \partial_{x} \left( A_0^{\frac32} \right)  \right)^T$.

%This special splitting in the source term of (\ref{eqn-1-modified}) is crucial for the design of the well-balanced WENO schemes
%for the steady state solutions described by (\ref{steady-1D}).

Subsequently, we consider a numerical scheme for solving (\ref{eqn-1-modified}). The scheme may be classified as a
   \emph{linear  scheme}, because all of the spatial derivatives are approximated by a linear finite difference
operator $D$ that is defined to satisfy
\begin{equation}  \label{linear-operator}
D(\alpha f + \beta g)=\alpha D(f)+\beta D(g)
\end{equation}
 \noindent for any constants $\alpha, \; \beta$ and grid functions $f$ and $g$.

For such a linear scheme, we have
\begin{proposition} \label{proposition-1}
A linear scheme for the 1D blood flow model satisfying the steady state solutions (\ref{steady-1D}) can maintain the well-balanced property.
\end{proposition}

\begin{proof}
For the steady state solutions (\ref{steady-1D}), linear schemes satisfying (\ref{linear-operator}) are exact for the first equation
$ \partial_{x}Q=0 $, since $Q=0$ due to $u=0$, and the truncation error for the second equation reduces to
$$
\begin{array} {l}
D\left(  \frac{Q^{2}}{A}+\frac{k}{3\rho \sqrt{\pi} } A^{\frac{3}{2}}   \right) - \frac{k}{\rho \sqrt{\pi}} (A - A_0) D(\sqrt{A_0} )
+ \frac{k}{3 \rho\sqrt{\pi}}  D \left( A_0^{\frac{3}{2}} \right) \\
= D\left( \frac{k}{3\rho \sqrt{\pi} } A^{\frac{3}{2}} - \frac{k}{3\rho \sqrt{\pi} } A_0^{\frac{3}{2}}   \right)  \\
= 0,
\end{array}
$$
where the first equality thanks to the facts that $Q=0$ due to $u=0$ and $A = A_0$ as well as the linearity of the finite difference operator $D$; the second one is also due to
the fact that $ A = A_0$ and the consistency of the finite difference operator $D$. As a consequence, this finishes the proof.
\end{proof}

However, the WENO schemes are nonlinear. The nonlinearity comes from
the nonlinear weight, which in turn comes from the nonlinearity of
the smoothness indicators. In order to construct a linear scheme
which can maintain the well-balanced property even with the presence of
the nonlinearity of the nonlinear weight and does not affect the high-order accuracy, we must take some modifications.

To present the basic ideas of the modification, we firstly consider the situation
when the WENO scheme is applied without the flux splitting and the
local characteristic decomposition.

Before considering an approximation of the flux gradient $f(U)_x$, we must firstly reconstruct the numerical flux
$\hat{f}_{j+1/2}$. We consider a WENO scheme with a global Lax-Friedrichs flux splitting, denoted by the WENO-LF scheme. Now the flux $f(U)$ writes
$$f(U) = f^+(U) + f^-(U),$$
where
\begin{equation}  \label{flux-splitting-original}
f^{\pm}(U)=\frac{1}{2}\left[ \left(\begin{array}{c}
Q \\
\frac{Q^{2}}{A}+\frac{k}{3\rho \sqrt{\pi} } A^{\frac{3}{2}}
\end{array}
\right) \pm \alpha _i\left(
\begin{array}{c}
A\\
Q
\end{array}
\right) \right],
\end{equation}
with
\begin{equation}  \label{alpha}
\alpha_i=\max\limits_u|\lambda_i(u)|
\end{equation}
for the $i$th characteristic field, where $\alpha_i = \max\limits_u |\lambda_i(u)|$ with $\lambda_i(u)$ being the $i$th eigenvalue of the Jacobian
$f'(U)$. In order to design a linear finite difference operator, we adopt a minor modification to the flux splitting by replacing
$\pm \alpha_i
\left(
\begin{array}{c}
A \\
Q
\end{array}
\right)$
in (\ref{flux-splitting-original}) with
$\pm \alpha_i
\left(
\begin{array}{c}
A - A_0\\
Q
\end{array}
\right)$.
So the flux splitting (\ref{flux-splitting-original}) now becomes
\begin{equation} \label{flux-splitting-modified}
f^{\pm}(U)=\frac{1}{2}\left[ \left(\begin{array}{c}
Q \\
\frac{Q^{2}}{A} + \frac{k}{3\rho \sqrt{\pi} } A^{\frac{3}{2}}
\end{array}
\right) \pm \alpha _i\left(
\begin{array}{c}
A - A_0\\
Q
\end{array}
\right) \right].
\end{equation}
This modification is justified by the fact that $A_0$ is independent of time $t$.

Provided $\hat{f}_{j+1/2}=\hat{f}_{j+1/2}^{+} + \hat{f}_{j+1/2}^{-}$ based on the WENO approximation procedure using the modified flux splitting (\ref{flux-splitting-modified}), the flux gradient $f(U)_x$ may be finally approximated by
$$f(U)_x \big|_{x=x_j} \approx  \frac{\hat{f}_{j+1/2}-\hat{f}_{j-1/2}}{\Delta x}.$$

Herein, in order to achieve a more accurate solution at the price of more complicated computations, the WENO
approximation is implemented with a local characteristic
decomposition procedure, see \cite{Shu1997} for more details.

Subsequently, the WENO-LF schemes can be demonstrated to maintain the steady state solutions (\ref{steady-1D}), \emph{i.e.,} to satisfy the well-balanced property.

Firstly, $\hat{f}^{+}_{j+1/2}$ is given by
\begin{equation}   \label{fp}
\begin{array}            {lcl}
\hat{f}^{+}_{j+1/2} &=& \sum\limits_{k=-r}^{r}c_kf^+_{j+k}\\
                    &=& \sum\limits_{k=-r}^{r}c_k\frac{1}{2}\left(f_{j+k} + \alpha U_{j+k} \right) \\
                    &=& \frac{1}{2}\sum\limits_{k=-r}^{r}c_kf_{j+k} + \frac{1}{2} \sum\limits_{k=-r}^{r}c_k\left(\alpha U_{j+k}\right),
\end{array}
\end{equation}
where $f^{+}=f^{+}(U)$ is defined in (\ref{flux-splitting-modified}) with
$U=(A - A_0, Q)^T$ and $f=f(U)=\left(Q,  \frac{Q^{2}}{A}+\frac{k}{3\rho \sqrt{\pi} } A^{\frac{3}{2}}   \right)^T$ being the vector
grid functions, $c_k$ is a $2\times2$ matrix depending nonlinearly
on the smoothness indicators of $f^{+}$ on the stencil
$\{x_{j-r},\ldots,x_{j+r}\}$, and $\alpha$ is a $2\times2$ diagonal
matrix involving $\alpha_i$ in (\ref{alpha}).

Similarly, $\hat{f}^{-}_{j+1/2}$ can be written as
\begin{equation}\label{fm}
\begin{array}{lcl}
\hat{f}^{-}_{j+1/2} &=& \sum\limits_{k=-r+1}^{r+1}a_kf^{-}_{j+k} \\
                    &=& \sum\limits_{k=-r+1}^{r+1}a_k\frac{1}{2}\left(f_{j+k} - \alpha U_{j+k} \right) \\
                    &=& \frac{1}{2}\sum\limits_{k=-r+1}^{r+1}a_kf_{j+k} - \frac{1}{2}
                            \sum\limits_{k=-r+1}^{r+1}a_k\left(\alpha U_{j+k}\right),
\end{array}
\end{equation}
where $f^{-}=f^{-}(U)$.  As $c_k$ in (\ref{fp}), herein $a_k$ is also a $2\times2$ matrix but depending nonlinearly
on the smoothness indicators of $f^{-}$ on the stencil
$\{x_{j-r+1},\ldots,x_{j+r+1}\}$, and $\alpha$ is a $2\times2$ diagonal
matrix involving $\alpha_i$ in (\ref{alpha}).

Ultimately, we have
\begin{equation} \label{f-plus}
\begin{array}         {lcl}
\hat{f}_{j+1/2}  &=&    \hat{f}_{j+1/2}^{+} + \hat{f}_{j+1/2}^{-}. \\
                 &=&         \frac{1}{2}\sum\limits_{k=-r}^{r}c_kf_{j+k} + \frac{1}{2}
\sum\limits_{k=-r}^{r}c_k\left(\alpha U_{j+k}  \right)
+
\frac{1}{2}\sum\limits_{k=-r+1}^{r+1}a_kf_{j+k} - \frac{1}{2}
\sum\limits_{k=-r+1}^{r+1}a_k\left(\alpha U_{j+k} \right).
\end{array}
\end{equation}

Likewise, $\hat{f}_{j-1/2}^{+}$ and $\hat{f}_{j-1/2}^{-}$ can be defined. So, we can obtain $\hat{f}_{j-1/2}$ as follows
\begin{equation}\label{f-minus}
\begin{array}  {lcl}
\hat{f}_{j-1/2} &=&   \hat{f}_{j-1/2}^{+} + \hat{f}_{j-1/2}^{-}. \\
                &=&   \frac{1}{2}\sum\limits_{k=-r-1}^{r-1}\hat{c}_kf_{j+k} + \frac{1}{2}
\sum\limits_{k=-r-1}^{r-1}\hat{c}_k \left(\alpha U_{j+k}\right) +
\frac{1}{2}\sum\limits_{k=-r}^{r} \hat{a}_k f_{j-k}  -  \frac{1}{2}\sum\limits_{k=-r}^{r} \hat{a}_k \left(\alpha U_{j-k}\right).
\end{array}
\end{equation}
Herein, $\hat{c}_k$ is a $2\times2$ matrix depending nonlinearly
on the smoothness indicators of $f^{+}$ on the stencil
$\{x_{j-r-1},\ldots,x_{j+r-1}\}$.  $\hat{a}_k$ is also a $2\times2$ matrix depending nonlinearly
on the smoothness indicators of $f^{-}$ on the stencil
$\{x_{j-r},\ldots,x_{j+r}\}$.

Subsequently, the approximation to $f(U)_x$ can be obtained
as follows
\begin{equation}\label{approximation-f}
\begin{array}{lcl}
f(U)_x \big|_{x=x_j}
&\approx& \frac{1}{\Delta
x}\left(\hat{f}_{j+1/2} -
\hat{f}_{j-1/2}\right)\\
&=&\frac{1}{\Delta x}\left[ \left(
\frac{1}{2}\sum\limits_{k=-r}^{r}c_kf_{j+k} + \frac{1}{2}
\sum\limits_{k=-r}^{r}c_k\left(\alpha U_{j+k}\right) +
\frac{1}{2}\sum\limits_{k=-r+1}^{r+1}a_kf_{j+k} - \frac{1}{2}
\sum\limits_{k=-r+1}^{r+1}a_k\left(\alpha U_{j+k}\right)
 \right)\right. \\
&-& \left.\left(
\frac{1}{2}\sum\limits_{k=-r-1}^{r-1}\hat{c}_kf_{j+k} + \frac{1}{2}
\sum\limits_{k=-r-1}^{r-1}\hat{c}_k \left(\alpha U_{j+k}\right) +
\frac{1}{2}\sum\limits_{k=-r}^{r} \hat{a}_kf_{j-k} -
                       \frac{1}{2}\sum\limits_{k=-r}^{r} \hat{a}_k \left(\alpha U_{j-k}\right)\right)
\right]
 \\
&=& \frac{1}{2\Delta x}\left(\sum\limits_{k=-r}^{r}c_kf_{j+k} -
\sum\limits_{k=-r-1}^{r-1}\hat{c}_kf_{j+k}\right)\\
&+& \frac{1}{2\Delta x}\left(\sum\limits_{k=-r+1}^{r+1}a_kf_{j+k} -
\sum\limits_{k=-r}^{r}\hat{a}_kf_{j-k}\right)\\
&+& \frac{1}{2\Delta x}\left(\sum\limits_{k=-r}^{r}c_k(\alpha
U_{j+k}) -
\sum\limits_{k=-r-1}^{r-1}\hat{c}_k(\alpha U_{j+k})\right)\\
&+& \frac{1}{2\Delta x}\left(\sum\limits_{k=-r}^{r} \hat{a}_k(\alpha
U_{j-k}) - \sum\limits_{k=-r+1}^{r+1}a_k(\alpha U_{j+k})\right).
\end{array}
\end{equation}

It should be noted that with
$\pm \alpha U= \pm \alpha
\left(A - A_0, \, Q \right)^T$
instead of
$\pm \alpha
\left(A, \, Q\right)^T$ in the flux splitting (\ref{flux-splitting-original}),
the first two terms on the right hand side of the above expression become constant vectors for the steady state solutions (\ref{steady-1D}). Denoting $U_{j+k}$ as $U$ for simplicity, we have $\alpha U_{j+k} = \alpha U$ as a constant vector. Consequently
\begin{equation}
\begin{array}{ll}
 & \frac{1}{2\Delta x}\left(\sum\limits_{k=-r}^{r}c_k\left(\alpha U_{j+k}\right)
- \sum\limits_{k=-r-1}^{r-1}\hat{c}_k\left(\alpha U_{j+k}\right)\right) \\ = &
\frac{1}{2\Delta x}\left(\sum\limits_{k=-r}^{r}c_k(\alpha U ) -
\sum\limits_{k=-r-1}^{r-1}\hat{c}_k(\alpha U )\right) \\
= & \frac{1}{2\Delta
x}\left[\left(\sum\limits_{k=-r}^{r}c_k\right)(\alpha U ) -
\left(\sum\limits_{k=-r-1}^{r-1}\hat{c}_k\right)(\alpha U )\right]
\\
= & \frac{1}{2\Delta x} \left[I \cdot(\alpha U) - I \cdot(\alpha U)\right] \\
= & 0,
\end{array}
\end{equation}
where $I$ is a $2\times2$ identity matrix, the identities
$\sum\limits_{k=-r}^{r}c_k = I$ and
$\sum\limits_{k=-r-1}^{r-1}\hat{c}_k = I$ are due to the consistency
of the WENO approximation. Similarly, we have
\begin{equation}
\frac{1}{2\Delta x}\left(\sum\limits_{k=-r}^{r}  \hat{a}_k\left(\alpha U_{j-k}\right) -
\sum\limits_{k=-r+1}^{r+1} a_k\left(\alpha U_{j+k}\right)\right) = 0.
\end{equation}

Finally, the approximation to $f(U)_x$ in
(\ref{approximation-f}) can be written as
\begin{equation}\label{Df}
\begin{array}{lcl}
f(U)_x\big|_{x=x_j}
&\approx&
\frac{1}{\Delta
x}\left(\hat{f}_{j+1/2} - \hat{f}_{j-1/2}\right)  \\
&=& \frac{1}{2\Delta x}\left(\sum\limits_{k=-r}^{r}c_k f_{j+k} -
\sum\limits_{k=-r-1}^{r-1}\hat{c}_kf_{j+k}\right)\\
&+& \frac{1}{2\Delta x}\left(\sum\limits_{k=-r+1}^{r+1}a_kf_{j+k} -
\sum\limits_{k=-r}^{r} \hat{a}_k f_{j-k}\right) \\
&=& \sum\limits_{k=-r-1}^{r+1}\beta_kf_{j+k} \\
&\equiv& D_f(f)_j,
\end{array}
\end{equation}
where $\beta_k$ is a $2\times2$ matrix depending on the smoothness
indicators involving $f^+$ and $f^-$. Motivated by the research work in \cite{Xing2005}, the key idea of the current scheme is to apply the finite difference operator $D_f$, with the smoothness indicators and the coefficient matrix $\beta_k$ in (\ref{Df}) fixed, to approximate the source terms $\left(0,   \sqrt{A_0} \right)_x^T$ and $\left(0, A_0^{\frac32}\right)_x^T$. This leads to the splitting of the two derivatives as
\begin{equation}\label{source-splitting}
\left(
\begin{array}{c}
0 \\
\sqrt{A_0 }
\end{array}
\right)_x = \frac{1}{2}\left(
\begin{array}{c}
0 \\
\sqrt{A_0}
\end{array}
\right)_x + \frac{1}{2}\left(
\begin{array}{c}
0 \\
\sqrt{A_0}
\end{array}
\right)_x ,
\quad
 \left(
\begin{array}{c}
0 \\
A_0^{\frac32}
\end{array}
\right)_x = \frac{1}{2}\left(
\begin{array}{c}
0 \\
A_0^{\frac32}
\end{array}
\right)_x + \frac{1}{2}\left(
\begin{array}{c}
0 \\
A_0^{\frac32}
\end{array}
\right)_x,
\end{equation}
which is handled by applying the similar flux splitting WENO approximation procedure. The two parts of each source term are approximated by
the finite difference operator $D_f$ with coefficients obtained from the computation of $f^+(U)$ and $f^-(U)$, respectively.

A key observation is that the finite difference operator $D_f$ in (\ref{Df}), with the fixed coefficient matrix $\beta_k$, is a \emph{linear} finite difference operator on any grid function as in (\ref{linear-operator}). In addition, the finite difference operator $D_f$ is a high order accurate \emph{linear} approximation to the first derivative of a grid function. Therefore based on the Proposition \ref{proposition-1}, it may be proved that the WENO scheme with the modified flux
 splitting (\ref{flux-splitting-modified}) and with the special handling of
 the source terms described in (\ref{source-splitting}) maintains
 the well-balanced property. This leads to

\begin{proposition} \label{proposition-2}
The WENO-LF scheme for the blood flow model satisfying the steady state solutions (\ref{steady-1D}) can maintain the well-balanced property without adverse effect on its original high order accuracy.
\end{proposition}

%\begin{remark} In this paper, we use the Lax-Friedrichs flux splitting only due to its simplicity. Other flux splitting, such as a Roe flux splitting and Roe flux %splitting with entropy fix, can be applied and also lead to a well-balanced finite difference WENO scheme. In fact, a well-balanced scheme is designed herein under the %linear scheme framework (see Proposition \ref{proposition-1}) and any Riemann solver may be implemented to obtain a linear scheme.
%\end{remark}

For the temporal discretization, high order total variation diminishing (TVD) Runge-Kutta methods \cite{Shu1988}
can be used. In the numerical section of this paper, we apply the third order Runge-Kutta methods:
\begin{equation}  \label{RK3}
\begin{array} {lcl}
U^{(1)} &=& \displaystyle U^n + \Delta t \mathcal{F}(U^n),                                                    \\
U^{(2)} &=& \displaystyle \frac{3}{4}U^n + \frac{1}{4}\left(U^{(1)} + \Delta t \mathcal{F}(U^{(1)}) \right),   \\
U^{n+1} &=& \displaystyle \frac{1}{3}U^n + \frac{2}{3}\left(U^{(2)} + \Delta t \mathcal{F}(U^{(2)}) \right),
\end{array}
\end{equation}
with $\mathcal{F}(U)$ being the spatial operator.

\section{ Numerical results}   \label{the-results}

In this section, we carry out extensive numerical experiments inspired by Delestre \emph{et al.} \cite{Delestre2013} to demonstrate the
performances of a fifth-order ($r=2$) finite difference WENO scheme. The $CFL$ number is taken as $0.6$, except for the accuracy tests where smaller time step is taken to ensure that spatial errors dominate.

\subsection{The ideal tourniquet} \label{ideal}

This example is similar to the Stoke's dam break problem in shallow water equations \cite{Delestre2013-SWE}.
Herein, we consider the analogous problem in blood flow model: a tourniquet is applied and we remove it instantaneously. And we consider the following initial conditions
$$
A(x,0)=\left\{
\begin{array}{ll}
\pi R_{_L}^2 & \mbox{if} \;  x \leq 0, \\
\pi R_{_R}^2 & \mbox{otherwise},
\end{array}
\right.
\; \; \mbox{and} \; \;\; Q(x,0) = 0,
$$
on a computational domain $[-0.04, 0.04]$ based on the following parameters: $ k=1.\times10^{7} \, Pa/m, \, \rho = 1060 \,\, kg/m^{3}, \,
 R_{_L}=5\times10^{-3} \, m,  \, R_{_R}=  4\times10^{-3} \,m$.

 We solve this example on the mesh with $200$ cells up to $t=0.005$ s and present the numerical solutions against the exact ones in \figurename  \, \ref{f:42-WB}. It is clear that the numerical results fit well with the exact ones and keep steep shock transitions.

\begin{figure}
\centering
\includegraphics[width=3.2in]{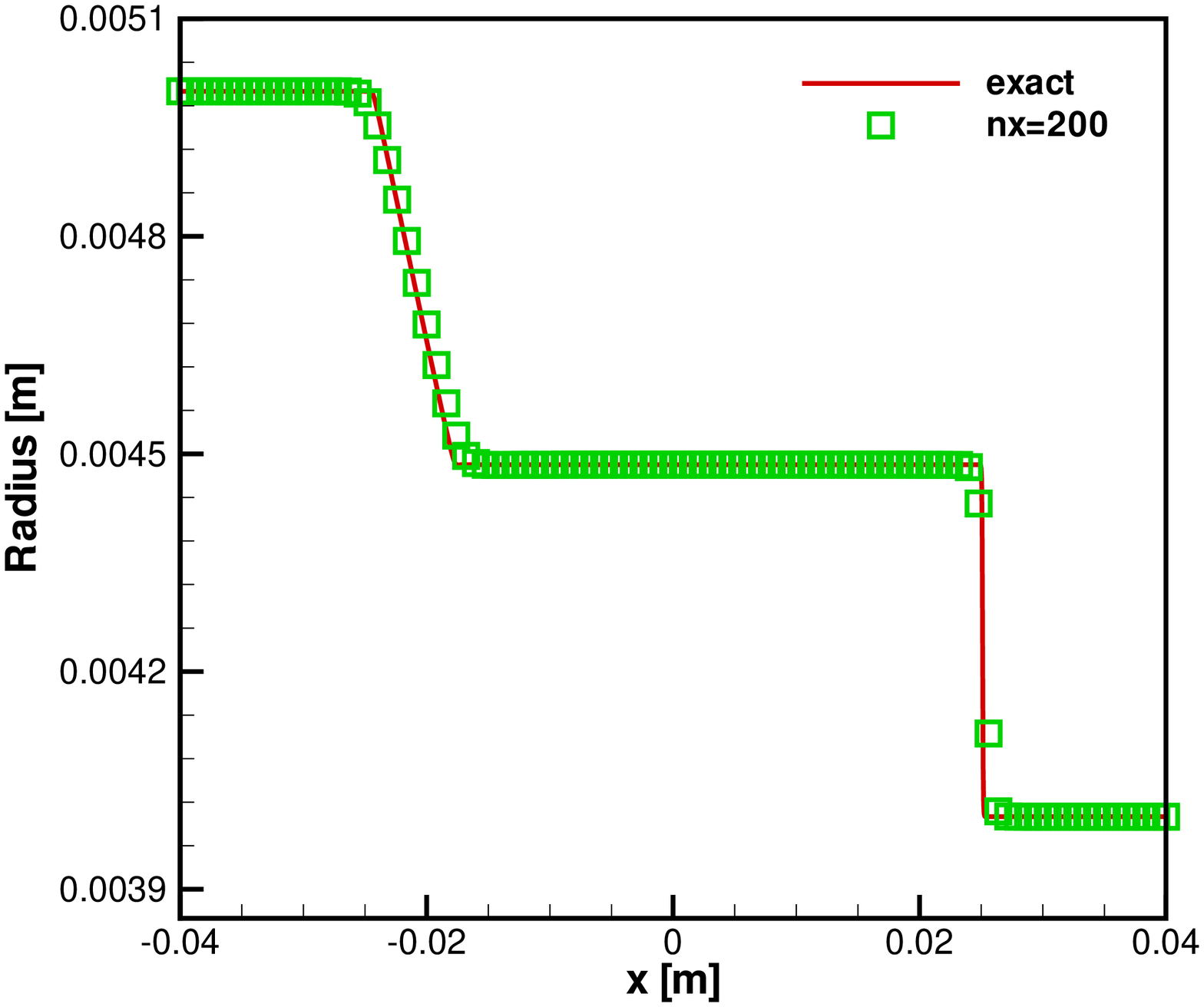}
\includegraphics[width=3.2in]{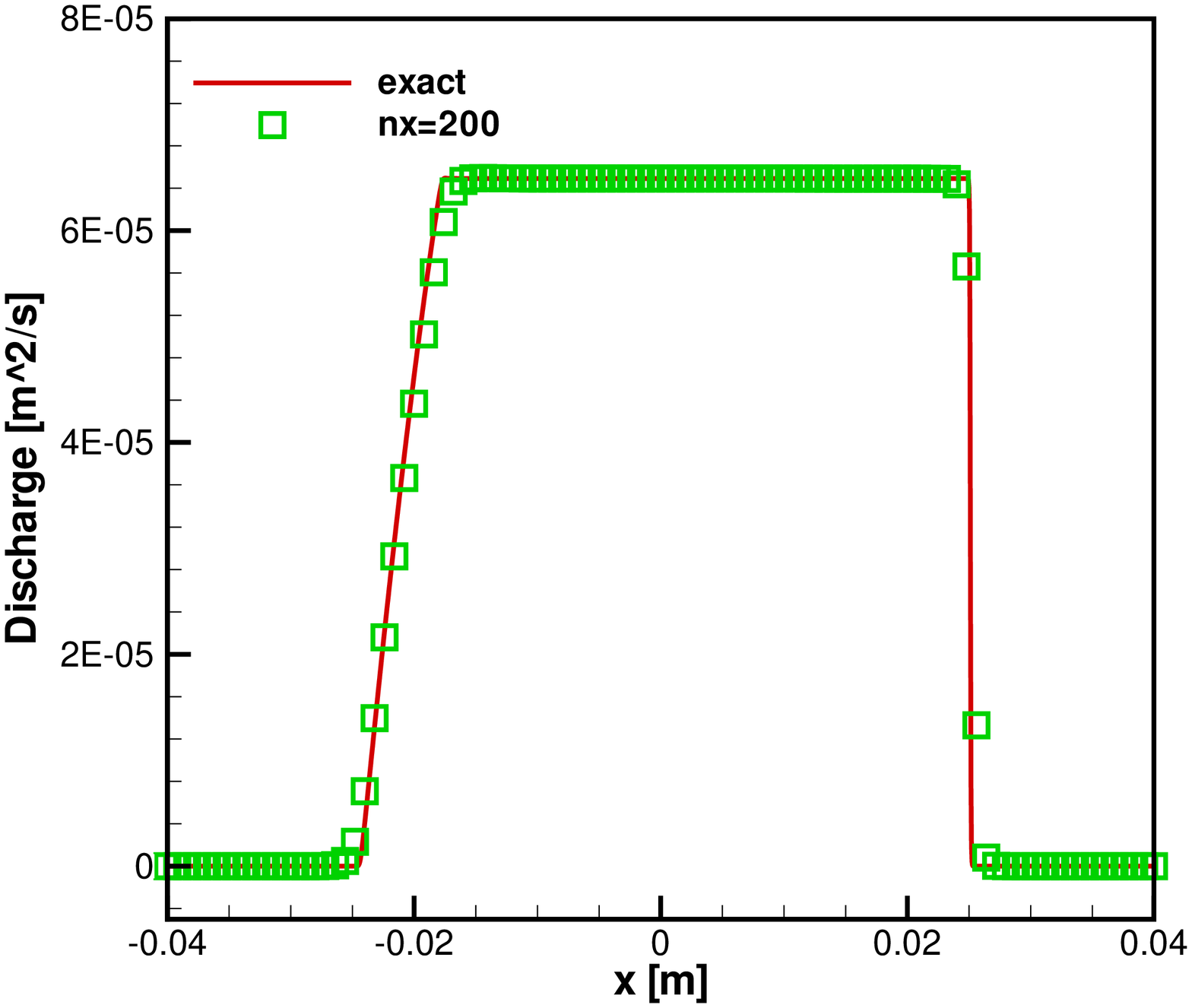}
\caption{ The numerical solutions of the ideal tourniquet problem in Section \ref{ideal} with $200$ cells at $t=0.005$ s. Radius (left) and discharge (right). } \label{f:42-WB}
\end{figure}

\subsection{ Wave equation }  \label{wave-eqn}

Then, the following quasi-stationary test case was proposed by Delestre \emph{et al.} \cite{Delestre2013}. It is chosen to demonstrate the
capability of the proposed scheme for computations on the perturbation of a steady state solutions. The initial conditions are given by
$$
A(x,0)=\left\{
\begin{array}{ll}
\pi R_{0}^{2} & \mbox{if} \;  x\in\left[0,\frac{40L}{100}\right] \cup \left[\frac{60L}{100},L\right], \\
\pi R_{0}^{2}\left[1 + \epsilon \sin\left(\pi\frac{x - 40L/100}{ 20L / 100}\right)\right]^{2} & \mbox{if} \;  x \in \left[ \frac{40L}{100}, \frac{60L}{100} \right],
\end{array}
\right.
\; \;
\mbox{and}
\; \;
Q(x,0) = 0,
$$
on the computational domain $[0, 0.16]$. The following parameters have been used for the example:  $ \epsilon = 5\times10^{-3}, \;   k = 10^{8} \, Pa/m, \;
\rho = 1060 \, kg/m^{3}, \; R_{0}=4\times10^{-3} \, m \, \mbox{and} \, L = 0.16 \, m.$

With the above initial conditions,  we obtain the following exact solutions:
$$
\left\{
\begin{array}{l}
R(x,t)=R_{0}+\frac{\epsilon}{2}\left[\Phi(x-c_{0}t)+\Phi(x+c_{0}t)\right],\\
u(x,t)=-\frac{\epsilon}{2}\frac{c_{0}}{R_{0}}\left[-\Phi(x-c_{0}t)+\Phi(x+c_{0}t)\right].
\end{array}
\right.
$$

We show the numerical solutions on a mesh with $200$ cells  at $t = 0.002 \,s,  \;  0.004 \, s, \; \mbox{and} \; 0.006\, s$, respectively in \figurename  \, \ref{f:43-WB}. The figure strongly suggests that the numerical solutions agree with the exact ones well.
Moreover, we also test the orders of the resulting scheme by plotting the numerical errors at $t=0.004 \, s$ and show the $L^1$ errors as well as order of accuracy for $A$ and $Q$  in Table \ref{t:order}. It is evident that the expected fifth order accuracy has been achieved.

\begin{figure}
\centering
%\flushmiddle
\includegraphics[width=3.2in]{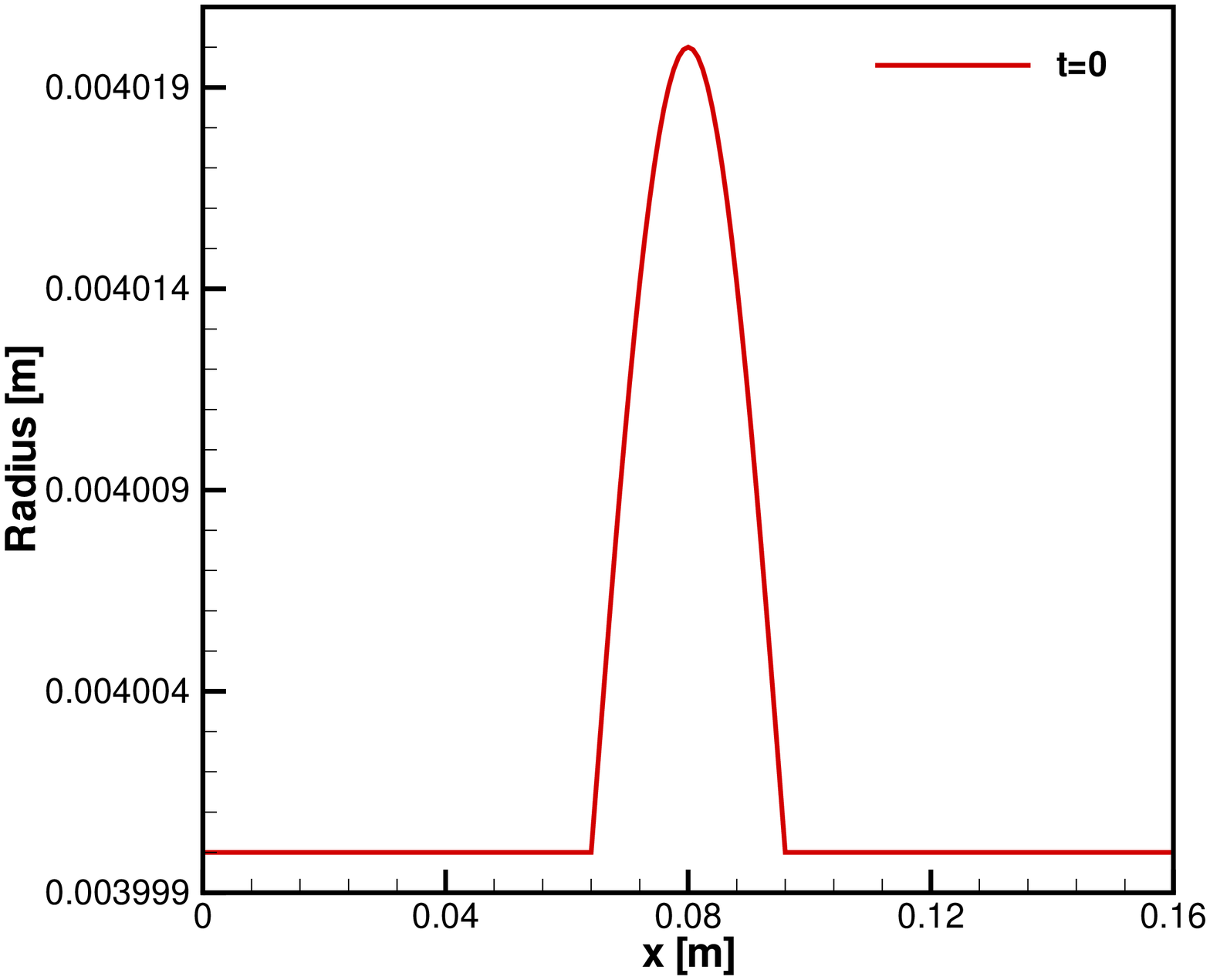}
\includegraphics[width=3.2in]{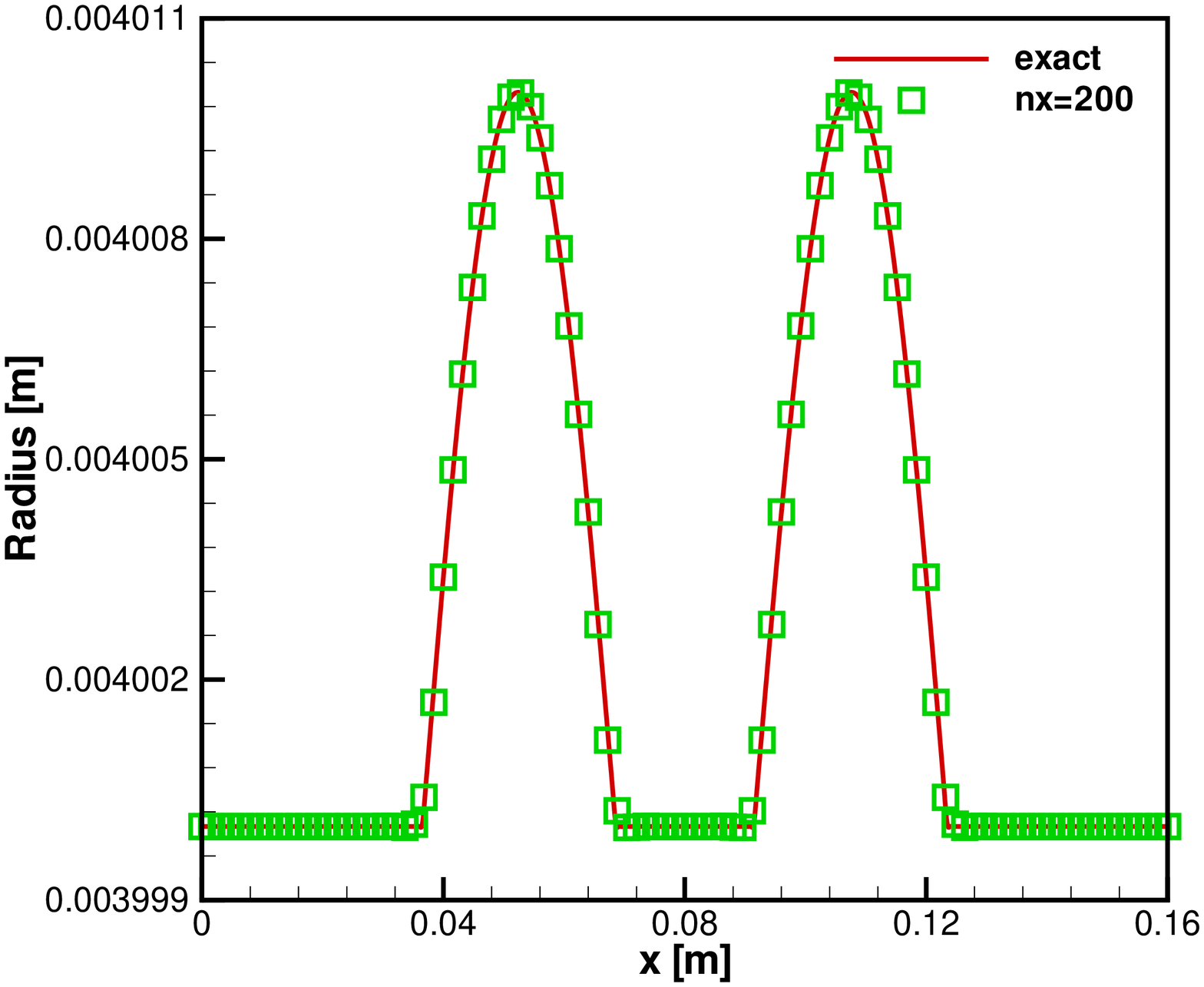}\\
\includegraphics[width=3.2in]{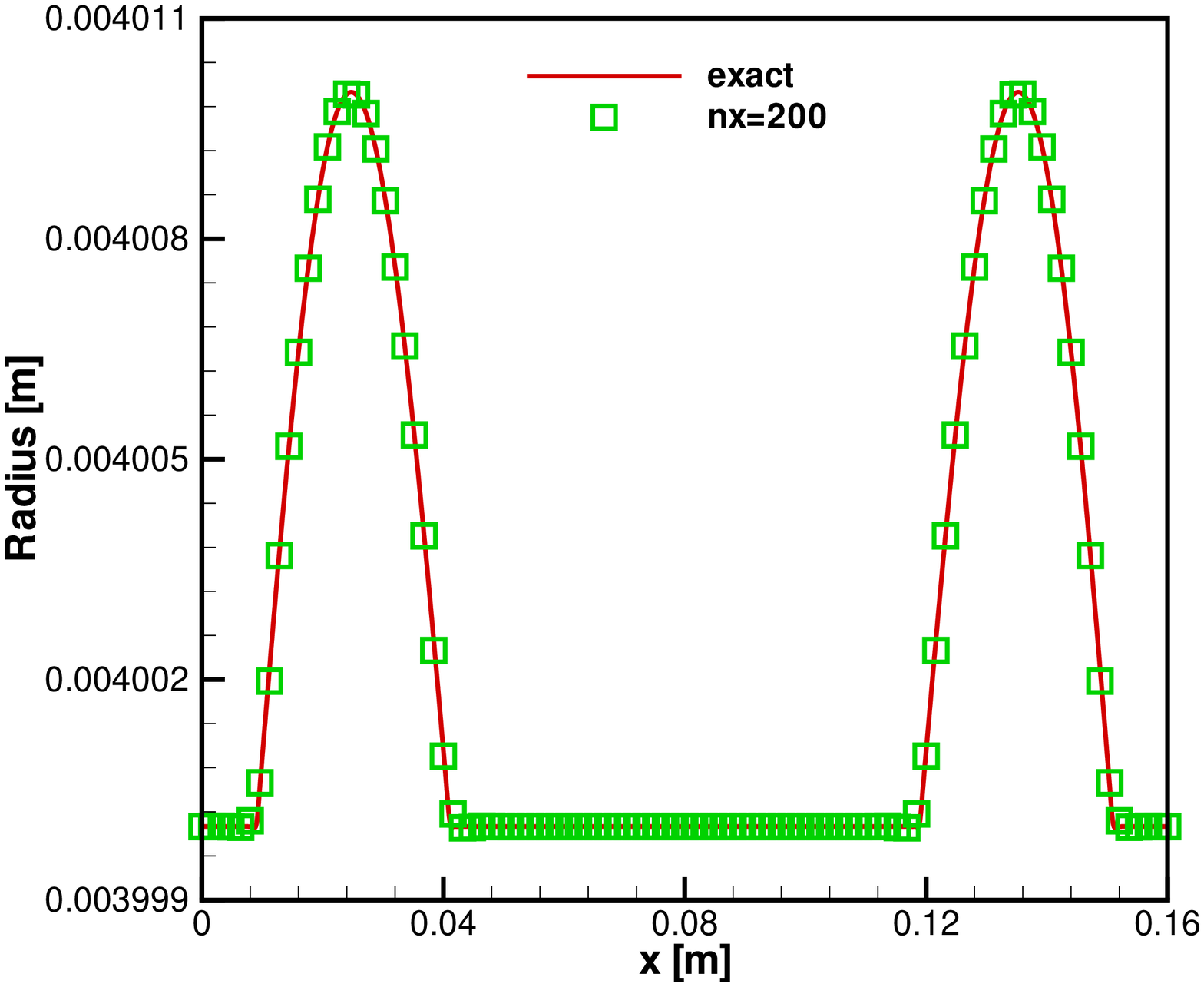}
\includegraphics[width=3.2in]{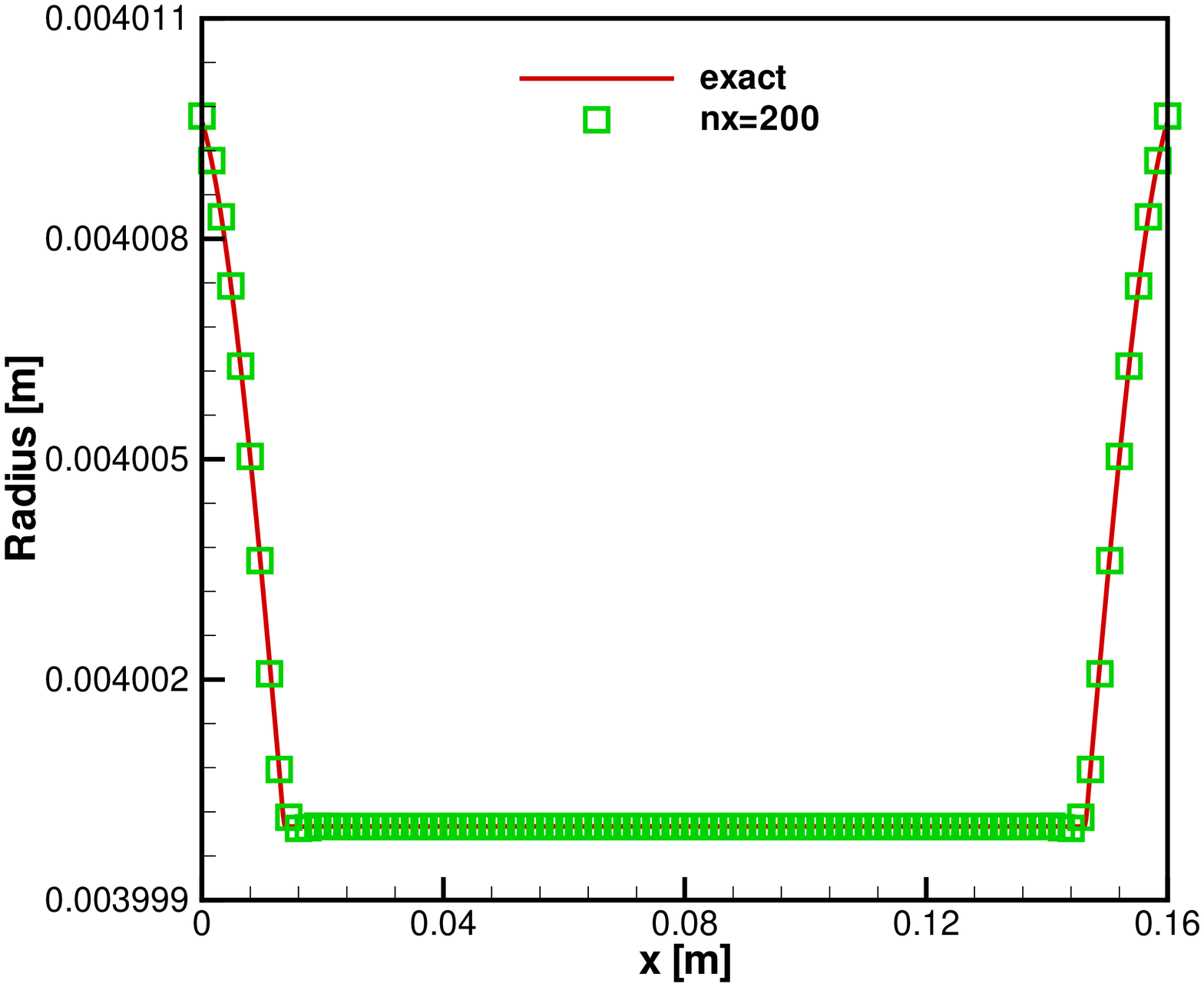}
\caption{ The numerical solutions of the wave equation problem in Section \ref{wave-eqn} with $ 200$ cells. Radius at time $t=0 \, s$ (upper left), $t=0.002 \, s$ (upper right), $t=0.004\, s$ (lower left), and $t=0.006 \, s$ (lower right), respectively. }  \label{f:43-WB}
\end{figure}

%\begin{table}
%\centering
%\caption{ $L^1$ errors and numerical orders of accuracy for the wave equation example of Section \ref{wave-eqn}.   } \label{t:order}
%\begin{tabular}{|l|*{4}{r|}}  \hline
%&\multicolumn{2}{c|}{$A$}&\multicolumn{2}{c|}{$Q$} \\\cline{2-5}
% \rule{0mm}{1.1em}\raisebox{3ex}[0ex]{N}&
%\multicolumn{1}{c|}{$L^{1} \; \mbox{error}$} &
%\multicolumn{1}{c|}{Order} & \multicolumn{1}{c|}{$L^{1} \;
%\mbox{error}$} & \multicolumn{1}{c|}{Order}  \\\hline
%    25 &   1.7566E-02 &          &   1.0990E-01 &            \\ \hline
%    50 &   2.2028E-03 &     3.00 &   1.9714E-02 & 2.48       \\ \hline
%   100 &   3.3138E-04 &     2.73 &   2.8273E-03 & 2.80       \\ \hline
%   200 &   2.3271E-05 &     3.83 &   2.0103E-04 & 3.81       \\ \hline
%   400 &   9.3899E-07 &     4.63 &   8.7320E-06 & 4.52       \\ \hline
%   800 &   3.1516E-08 &     4.90 &   3.7319E-07 & 4.55       \\ \hline
%  1600 &   9.1264E-10 &     5.11 &   1.1501E-08 & 5.02       \\ \hline
%\end{tabular}
%\end{table}

\begin{table}
\centering
\caption{ $L^1$ errors and numerical orders of accuracy for the wave equation example of Section \ref{wave-eqn}.   } \label{t:order}
\begin{tabular}{l*{4}{r}}  \hline
&\multicolumn{2}{c}{$A$}&\multicolumn{2}{c}{$Q$} \\   \cline{2-5}
 \rule{0mm}{1.1em}\raisebox{3ex}[0ex]{N}& \multicolumn{1}{c}{$L^{1} \; \mbox{error}$} & \multicolumn{1}{c}{Order} & \multicolumn{1}{c}{$L^{1} \;
\mbox{error}$} & \multicolumn{1}{c}{Order}  \\\hline
    25 &   1.7566E-02 &          &   1.0990E-01 &            \\
    50 &   2.2028E-03 &     3.00 &   1.9714E-02 & 2.48       \\
   100 &   3.3138E-04 &     2.73 &   2.8273E-03 & 2.80       \\
   200 &   2.3271E-05 &     3.83 &   2.0103E-04 & 3.81       \\
   400 &   9.3899E-07 &     4.63 &   8.7320E-06 & 4.52       \\
   800 &   3.1516E-08 &     4.90 &   3.7319E-07 & 4.55       \\
  1600 &   9.1264E-10 &     5.11 &   1.1501E-08 & 5.02       \\ \hline
\end{tabular}
\end{table}

\subsection{The man at eternal rest}  \label{WB}

The purpose of this example is to verify that the scheme indeed maintains the well-balanced property.

Herein, we consider a configuration with no flow and with a change of radius $R_{0}(x)$, this is the case for
a dead man with an aneurism. Thus, for the initial conditions, the section of the artery is not constant with the following form
$$
R(x,0)=R_{0}(x)=\left\{
  \begin{array}{ll}
   \tilde{R}                                                                           & \mbox{if} \; x \in [0,x_{1}]\cup[x_{4},L],\\
   \tilde{R} + \frac{\Delta R}{2}\left[\sin\left(\frac{x-x_{1}}{x_{2}-x_{1}}\pi-\frac{\pi}{2}\right)+1\right]  & \mbox{if} \; x \in [x_{1},x_{2}],\\
   \tilde{R} + \Delta R                                                                 & \mbox{if} \; x \in [x_{2},x_{3}],\\
   \tilde{R} + \frac{\Delta R}{2}\left[\cos\left(\frac{x-x_{3}}{x_{4}-x_{3}}\pi\right)+1\right]                & \mbox{if} \; x \in [x_{3},x_{4}],
   \end{array}
  \right.$$
on the computational domain $[0, L]$ with $ \tilde{R} = 4.\times10^{-3}\,  m, \Delta R=10^{-3} \, m, \, k=10^{8} \, Pa/m, \, \rho=1060\, kg/m^{3}, \,
L = 0.14 \, m, \, x_{1}= 10^{-2} \, m, \, x_{2}=3.05\times10^{-2} \, m, \, x_{3}=4.95\times10^{-2} \,  m  \,\, \mbox{and}  \,\, x_{4}=7.\times10^{-2}  \, m$.
In addition, the initial velocity is zero.  We compute this example up to $t = 5 \, s$.

In order to show that the well-balanced property is maintained  up to machine round off error,
tests are run using single, double and quadruple precisions, respectively. The $L^1$ and $L^{\infty}$ errors calculated for $A$ and $Q$ are presented in Table \ref{t:well-balanced}. It can be clearly seen that the $L^1$ and $L^{\infty}$
errors are all at the level of round off error associated with different precisions, which verify that the current scheme maintains
the well-balanced property as expected.

In \figurename  \, \ref{f:44-WB}, we present the radius at $ t = 5 \, s$ on a mesh with $ 200$
cells against a reference solution obtained with a much refined $2000$ cells. In addition, we run the same numerical test
using the non-well-balanced WENO schemes, with a straightforward integration of the source
term, and show their results in \figurename  \, \ref{f:44-WB} for comparison. It is obvious that the results of the well-balanced
WENO scheme are in good agreement with the reference solutions for the case, while the non-well-balanced WENO scheme fails to capture the small perturbation with $200$ cells.

\begin{table}
\centering
\caption{$L^1$ and $L^{\infty}$ errors for different precisions for the man at eternal rest.}  \label{t:well-balanced}
\begin{tabular}{l*{4}{r}}\hline
& \multicolumn{2}{c}{$L^1  \; \mbox{error}$}&\multicolumn{2}{c}{$L^{\infty}  \; \mbox{error}$} \\\cline{2-5}
\rule{0mm}{1.1em}\raisebox{3ex}[0ex]{Precision}&
\multicolumn{1}{c}{$A$} & \multicolumn{1}{c}{$Q$}   & \multicolumn{1}{c}{$A$} & \multicolumn{1}{c}{$Q$}   \\ \hline
Single    & 3.47e-07 & 3.13e-07   & 3.54e-07 & 3.25e-07     \\
Double    & 2.72e-16 & 4.34e-16   & 2.11e-15 & 3.14e-16     \\
Quadruple & 2.31e-33 & 4.34e-32   & 1.28e-33 & 2.44e-31     \\ \cline{1-5}
\end{tabular}
\end{table}

\begin{figure}
\centering
\includegraphics[width=4in]{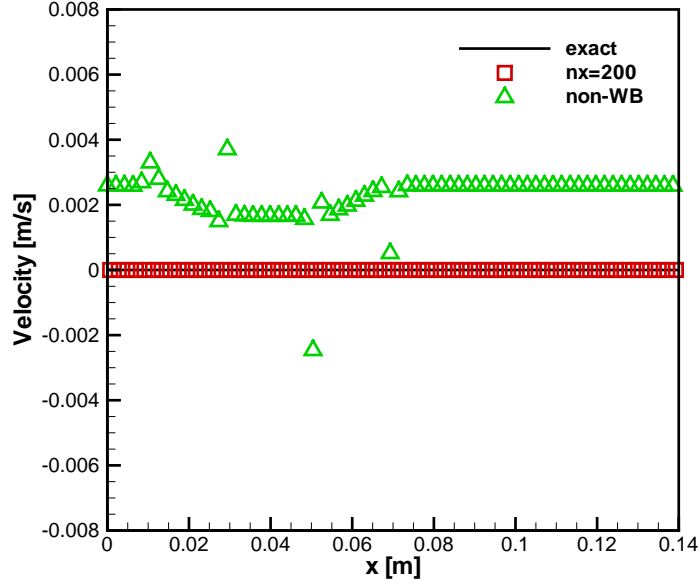}
\caption{ The man at eternal rest problem in Section \ref{WB} with $200$ cells at $t=5 \, s$. The result of the
well-balanced scheme with $200$ and $2000$ cells, and that of the non-well-balanced (denoted
by non-WB) scheme with $200$ cells. } \label{f:44-WB}
\end{figure}

\subsection{Propagation of a pulse to an expansion} \label{to-pulse}

Firstly, we test the case of a pulse in a section
$R_{_R}$ passing through an expansion: $A_{L}>A_{R}$, taking the following parameters: $k=1.0\times10^{8} \, Pa/m, \, L = 0.16 \, m,  \rho=1060 \, kg/m^{3}, \, R_{_L}=5\times10^{-3} \, m, \, R_{_R}=4\times10^{-3} \, m, \,
\Delta R=1.0\times10^{-3} \, m.$  We take a decreasing shape on a rather small scale:
$$
  R_{0}(x)=\left\{
  \begin{array}{ll}
   R_{_R}+ \Delta R  & \mbox{if} \; x \in [0,x_{1}],\\
   R_{_R}+ \frac{\Delta R}{2}\left[1+\cos\left(\frac{x-x_{1}}{x_{2}-x_{1}}\pi\right)\right] & \mbox{if} \; x \in [x_{1},x_{2}],\\
   R_{_R} &  \mbox{else},
   \end{array}
  \right.
$$
with $x_{1}=\frac{19L}{40}, \; x_{2}=\frac{L}{2}$. As initial conditions, we consider a fluid at rest $Q(x,0)=0 \, m^3/s$
and the following perturbation of radius:
$$
R(x,0)=\left\{
\begin{array}{ll}
  R_{0}(x)\left[1 + \epsilon \sin\left(\frac{100}{20L}\pi(x-\frac{65L}{100})\right)\right] & \mbox{if} \; x \in \left[\frac{65L}{100}, \frac{85L}{100}\right],\\
  R_{0}(x) & \mbox{else},
   \end{array}
  \right.$$
with $\epsilon=5.0\times10^{-3}$.

In \figurename  \, \ref{f:to-pulse}, we present the numerical results against the reference solutions at $t=0.002 \, s$ and $t=0.006 \, s$.
The numerical solutions are in good agreement with the reference ones and are comparable with those in \cite{Delestre2013}.

\begin{figure}
\centering
\includegraphics[width=3.2in]{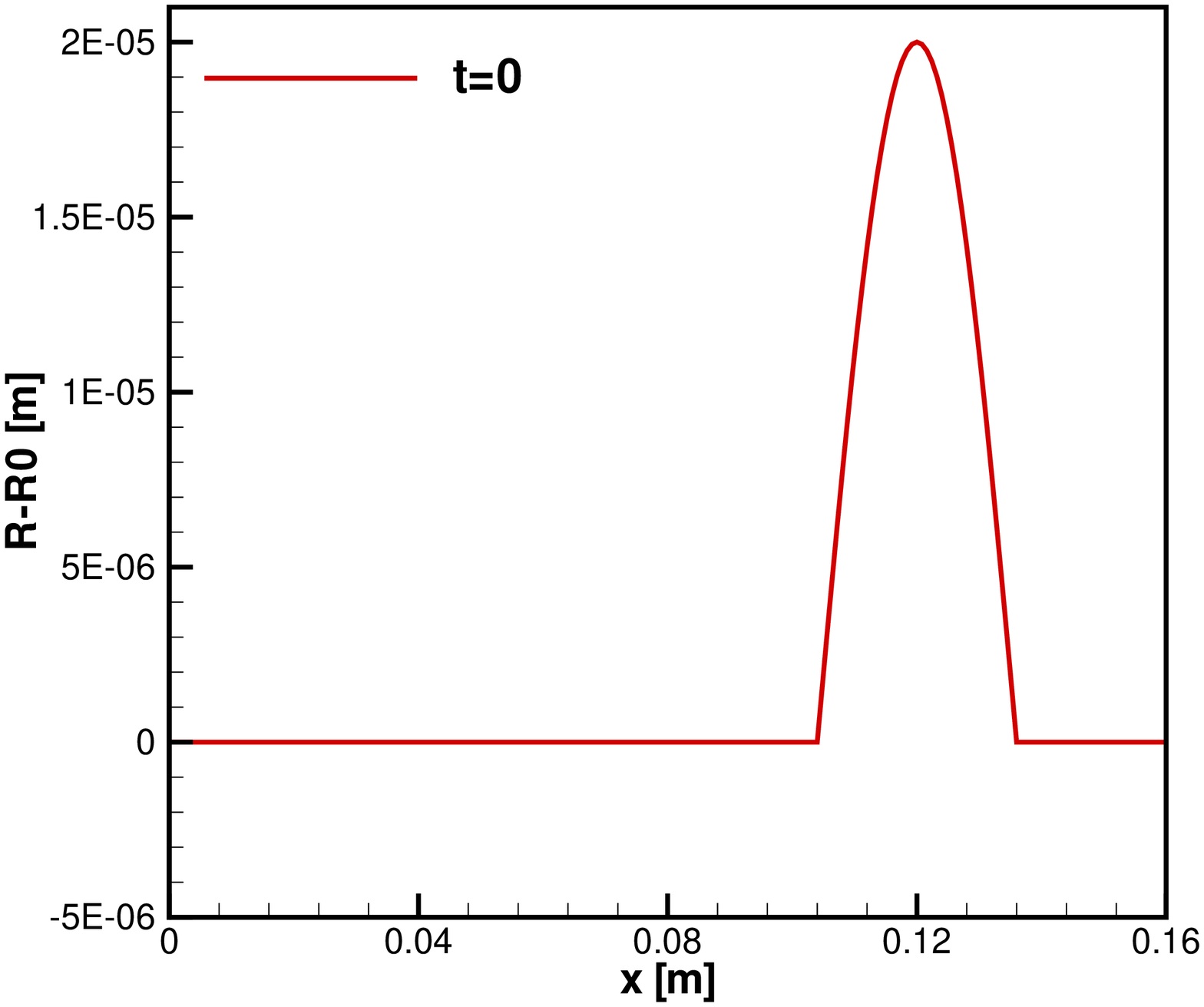}\\
\includegraphics[width=3.2in]{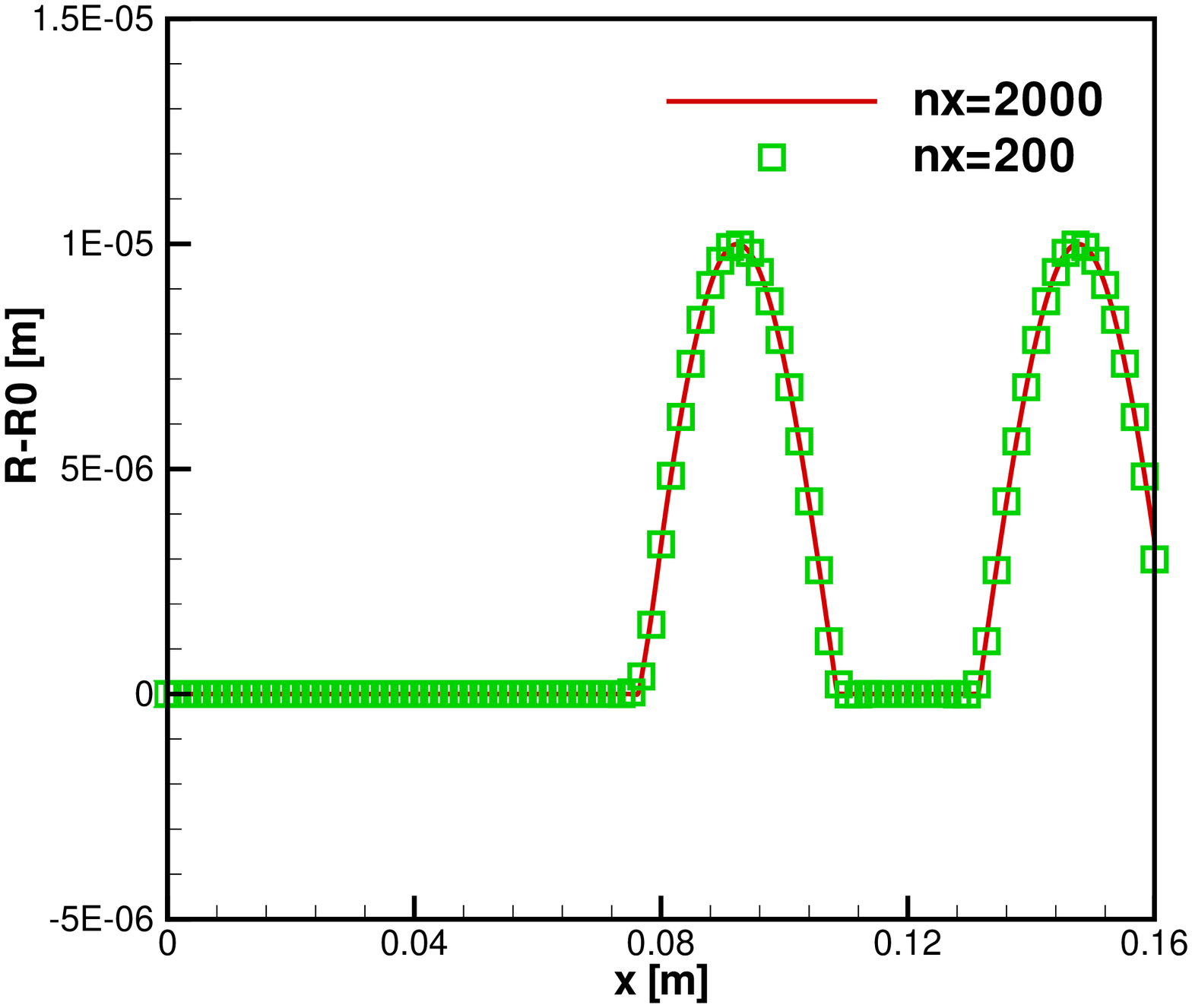}
\includegraphics[width=3.2in]{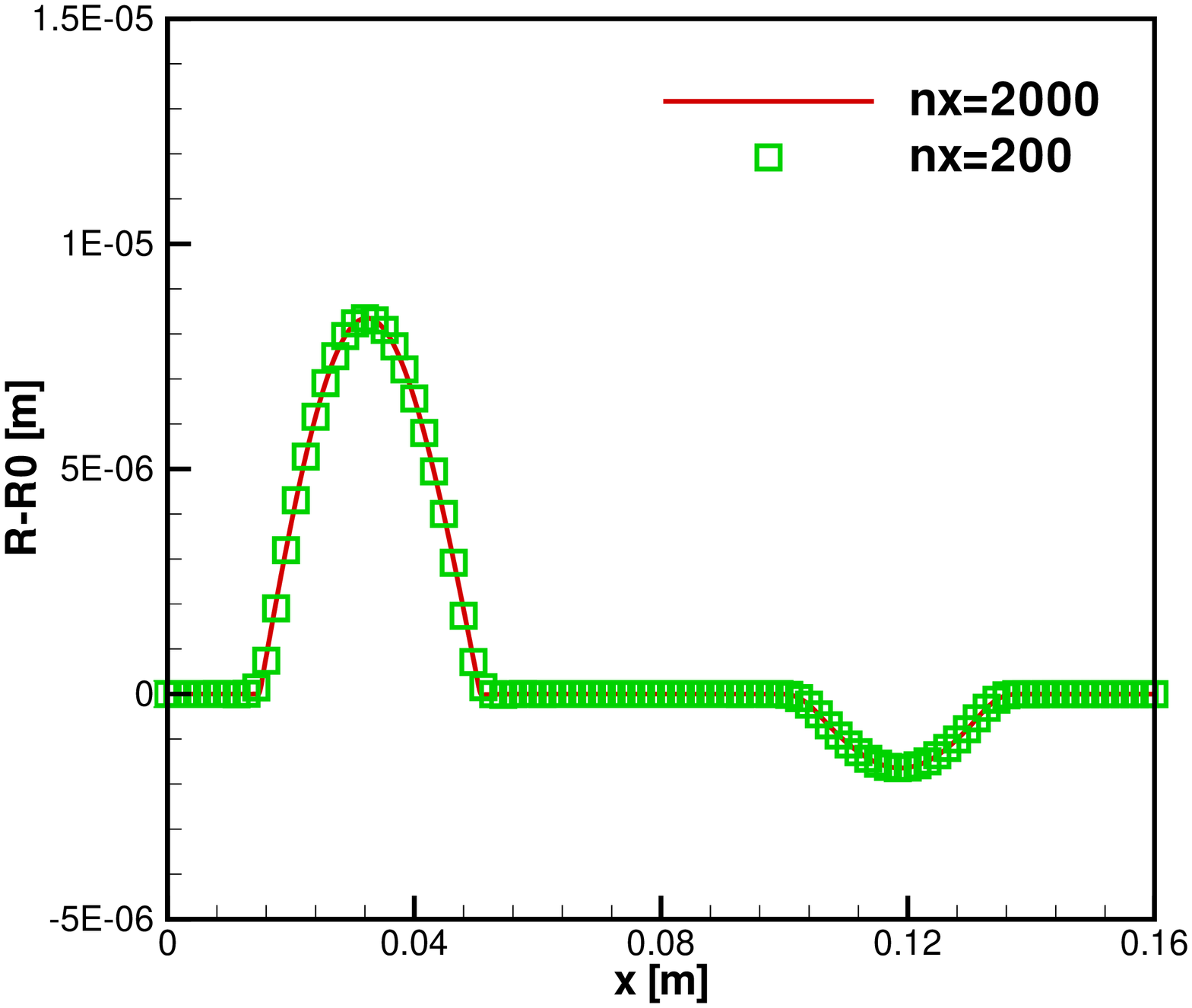}
\caption{ The numerical solutions of the propagation of a pulse to an expansion in Section \ref{to-pulse} with $ 200$ cells.
The errors $R - R_0$ at $t=0 \, s$ (upper), \, $t=0.002 \, s$ (lower left)  \, and \, $t=0.006\, s$  (lower right).  } \label{f:to-pulse}
\end{figure}

\subsection{Propagation of a pulse from an expansion} \label{from-pulse}

Then, we consider a pulse propagating from an expansion. So, the parameters are the same as in the Section \ref{to-pulse},
only the initial radius is changed:
$$
R(x,0)=
\left\{
\begin{array}{ll}
  R_{0}(x)\left[1+ \epsilon \sin\left(\frac{100}{20L}\pi\left(x-\frac{15L}{100}\right)\right)\right]  & \mbox{if} \;  x \in \left[\frac{15L}{100}, \frac{35L}{100}\right],\\
  R_{0}(x)  & \mbox{ else},
   \end{array}
  \right.$$
with $ \epsilon=5.0\times10^{-3}$.

In \figurename  \, \ref{f:from-pulse}, we demonstrate the numerical results against the reference solutions at $t=0.002 \, s$ and $t=0.006 \, s$.
Similar, the numerical solutions fit well with the reference ones and are comparable with those in \cite{Delestre2013}.

\begin{figure}
\centering
\includegraphics[width=3.2in]{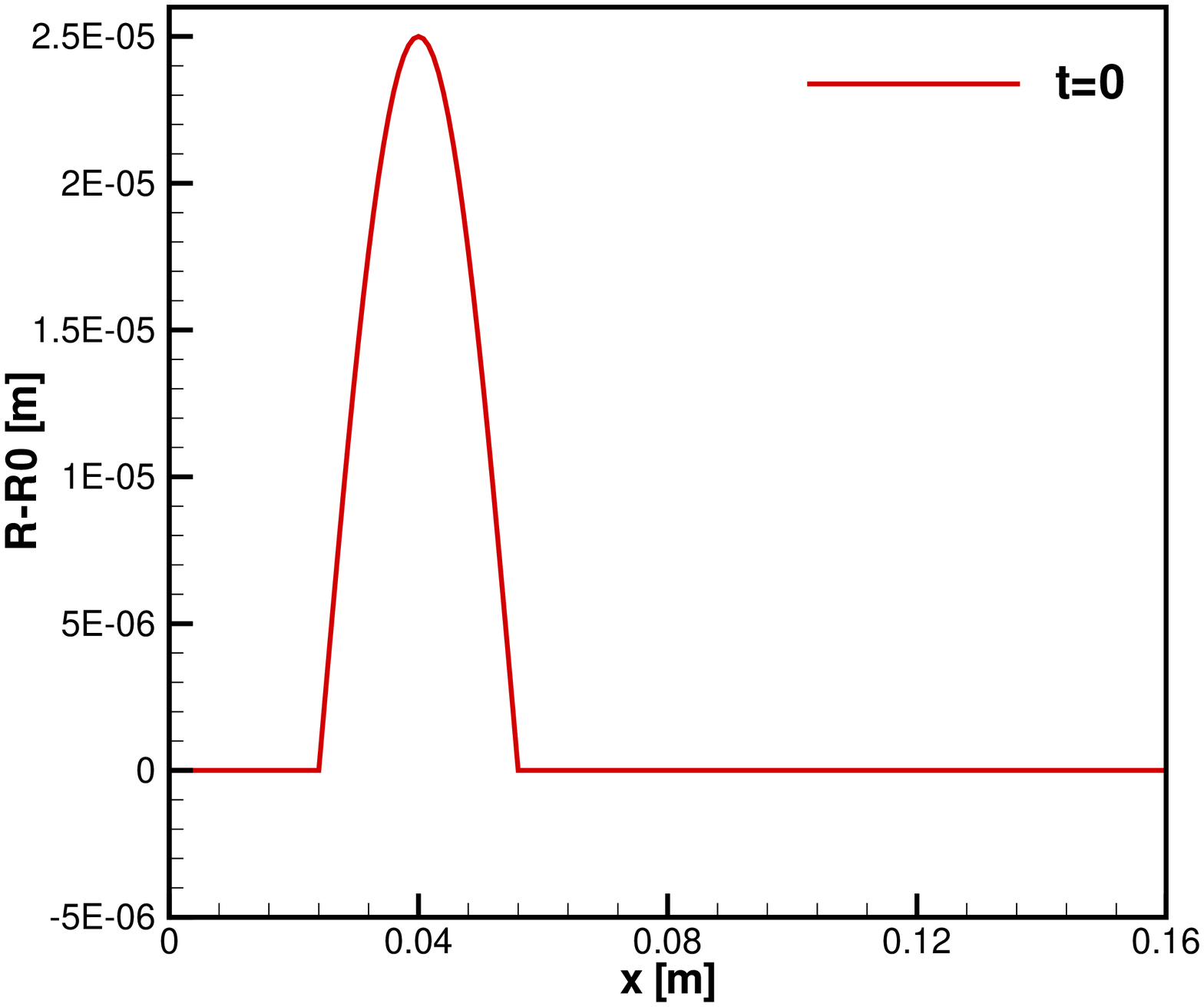} \\
\includegraphics[width=3.2in]{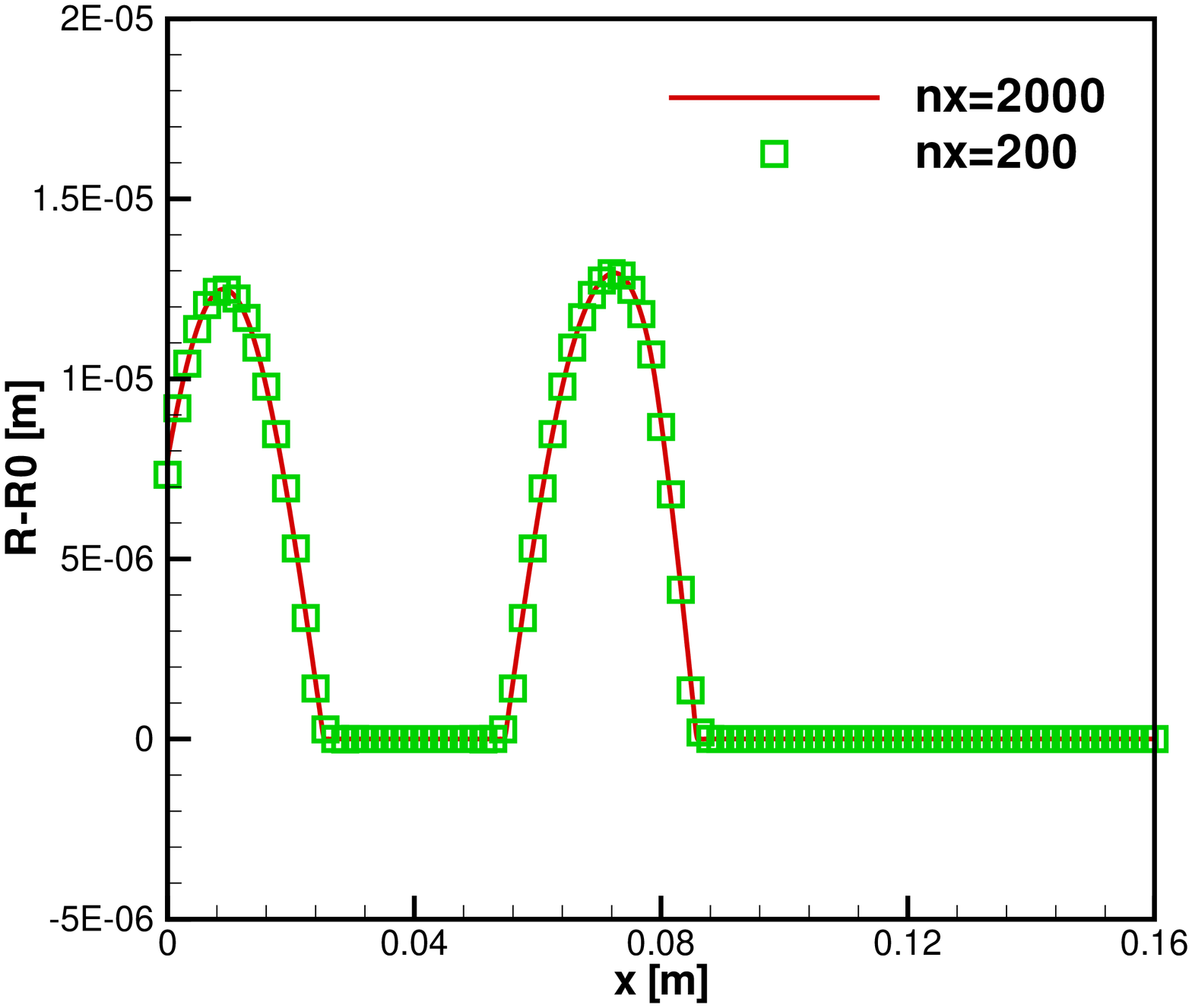}
\includegraphics[width=3.2in]{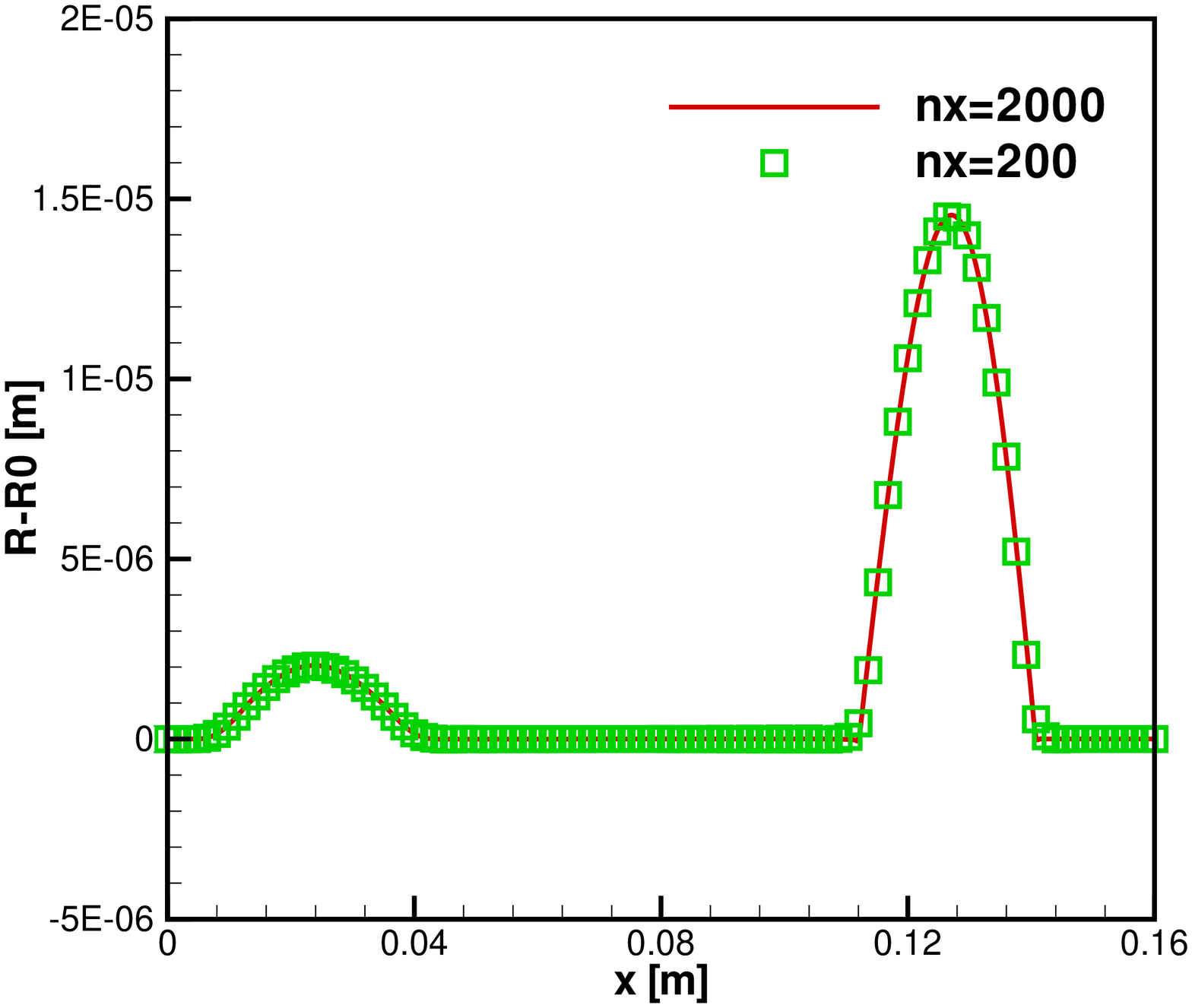}
\caption{ The numerical solutions of the propagation of a pulse from an expansion in Section \ref{from-pulse} with $ 200$ cells.
The errors $R - R_0$ at $t=0 \, s$  (upper), \, $t=0.002 \, s$ (lower left)  \, and \, $t=0.006\, s$ (lower right).    }   \label{f:from-pulse}
\end{figure}

\subsection{Wave damping}   \label{wave-damping}
In this last test case, we look at the viscous damping term in the linearized momentum equation. This
is the analogue of the Womersley \cite{Womersley1955} problem, we consider a periodic signal at the inflow with a constant section at rest.
 We consider the following model coupled with the linear friction term
\begin{equation}
\left\{
\begin{array}{l}
\partial_{t}A+\partial_{x}Q=0,\\
\partial_{t}Q+\partial_{x}\left(\frac{Q^{2}}{A}+\frac{k}{3\rho \sqrt{\pi} } A^{\frac{3}{2}}  \right) =
\frac{kA}{\rho \sqrt{\pi}}\partial_{x} ( \sqrt{ A_{0} } ) - C_f \frac{Q}{A},
\end{array}
\right.
\end{equation}
where $C_f = 8 \pi \nu$ with $\nu$ being the blood viscosity. We consider this example on the computational domain $[0, 3]$ subject to the given initial conditions
$$
\left\{
\begin{array}  {l}
A(x, 0) = \pi R_0^2, \\
Q(x,0)  = 0,
\end{array}
\right.
$$
companied by the following parameters: $k = 1 \times 10^8 \, Pa/m, \, \rho = 1060 \, kg/m^3, \, R_0 = 4 \times 10^-3 \, m. $ We solve this example up to $t= 25 \, s$.

Subsequently, we obtain a damping wave in the domain \cite{Delestre2013}
  \begin{equation} \label{model8}
  Q(t,x)=\left\{
  \begin{array}{ll}
  0                                             & \mbox{if} \; k_{r}x >    \omega t,\\
  Q_{_{\text{amp}}}\sin(\omega t-k_{r}x)e^{k_{i}x} & \mbox{if} \; k_{r}x \leq \omega t,
   \end{array}
  \right.
\end{equation}
with
$$
  \begin{array}{lcl}
k_{r} &=&  \left[\frac{\omega^{4}}{c_{0}^{4}} + \left(\frac{\omega C_{f}}{\pi R_{0}^{2}c_{0}^{2}}\right)^{2}\right]^{\frac{1}{4}}\cos\left(\frac{1}{2}\arctan\left(-\frac{C_{f}}{\pi R_{0}^{2}\omega}\right)\right),  \\
k_{i} &=& \left[\frac{\omega^{4}}{c_{0}^{4}} + \left(\frac{\omega C_{f}}{\pi R_{0}^{2}c_{0}^{2}}\right)^{2}\right]^{\frac{1}{4}}\sin\left(\frac{1}{2}\arctan\left(-\frac{C_{f}}{\pi R_{0}^{2}\omega}\right)\right), \\
\omega &=& 2 \pi/T_{_\text{pulse}}  = 2 \pi/0.5 \, s,  \\
c_0   &=& \sqrt{\frac{k \sqrt{A_0}}{2 \rho \sqrt{\pi}}} = \sqrt{\frac{k R_0}{2 \rho }}.
   \end{array}
$$

 For the boundary conditions, we impose the incoming discharge
  $$Q_{b}(t)=Q_{_{\text{amp}}}\sin(\omega  t) \, m^{3}/s,$$
 at $x=0 \, m$ with $Q_{_\text{amp}} = 3.45 \times 10^{-7} \, m^3/s^3$ being the amplitude of the inflow discharge.
 As the flow is subcritical, the discharge is imposed at the outflow boundary, thanks to (\ref{model8}) at the right boundary $ x= 3 \, m$.

In \figurename  \, \ref{f:damping}, we present the numerical results against the exact solutions at $t=25\, s$ with different $C_f$.
It is obvious that the numerical solutions are in good agreement with the exact solutions and are comparable with those in \cite{Delestre2013}.

\begin{figure}
\centering
\includegraphics[width=3.1in]{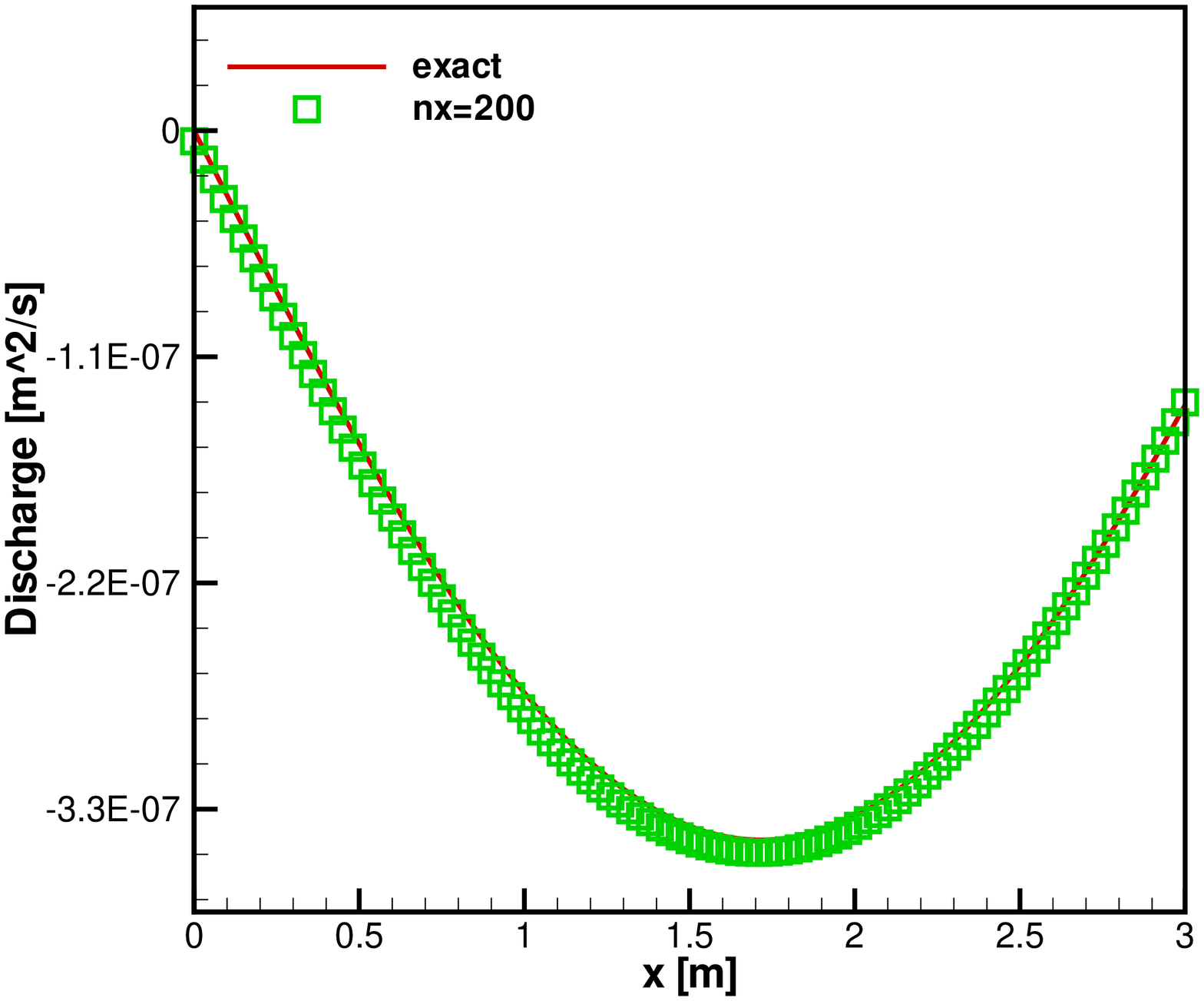}
\includegraphics[width=3.1in]{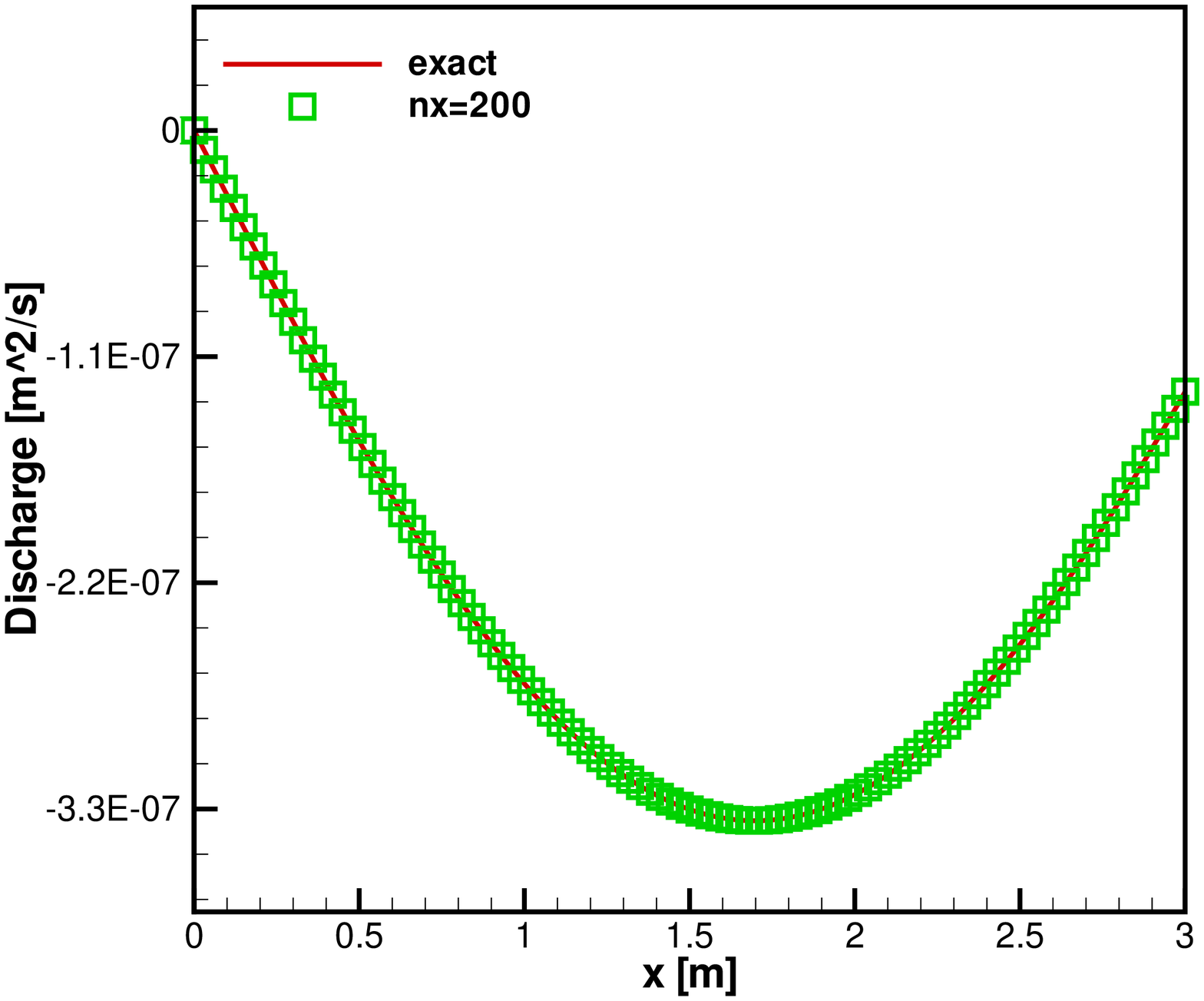}   \\
\includegraphics[width=3.1in]{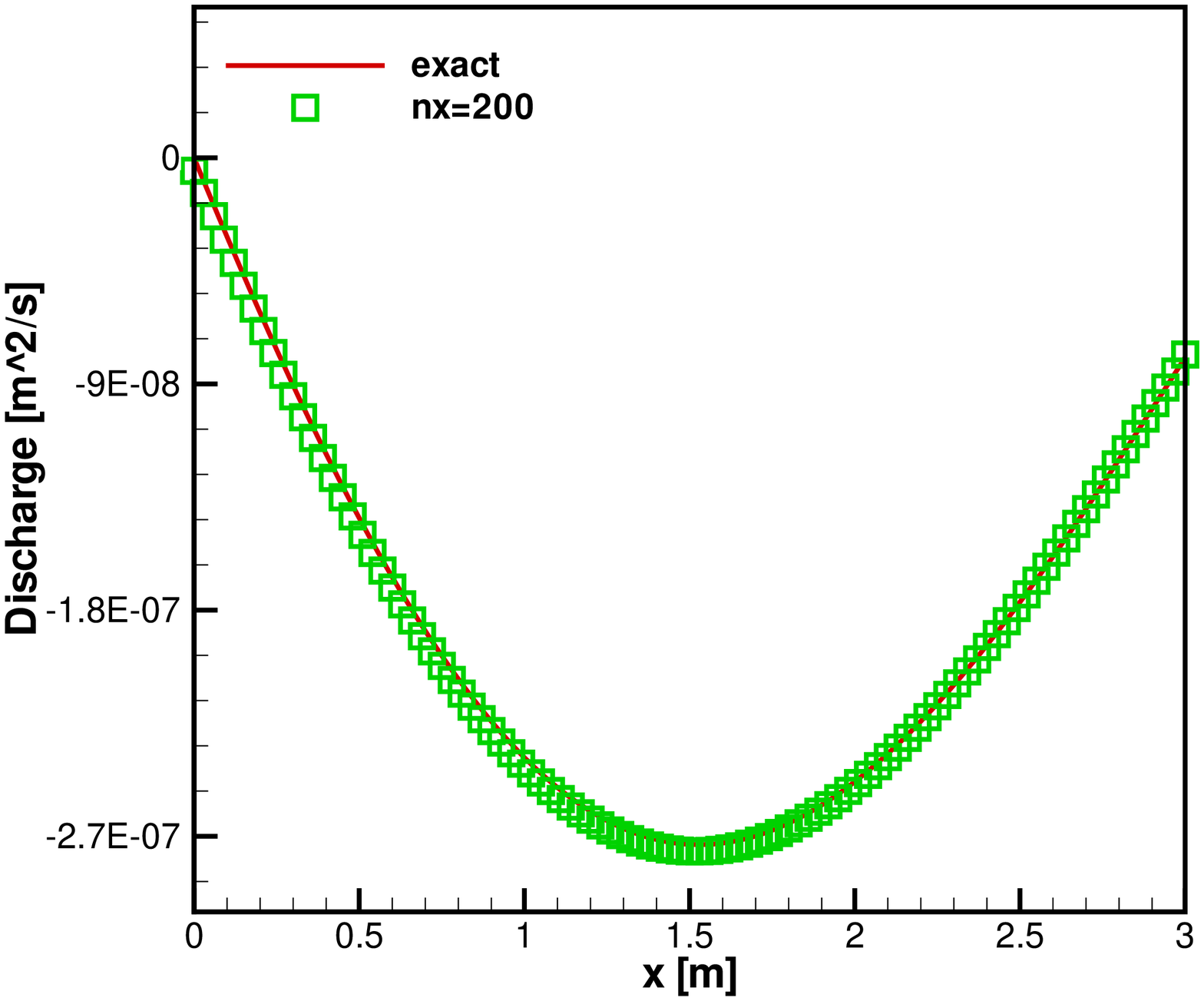}
\includegraphics[width=3.1in]{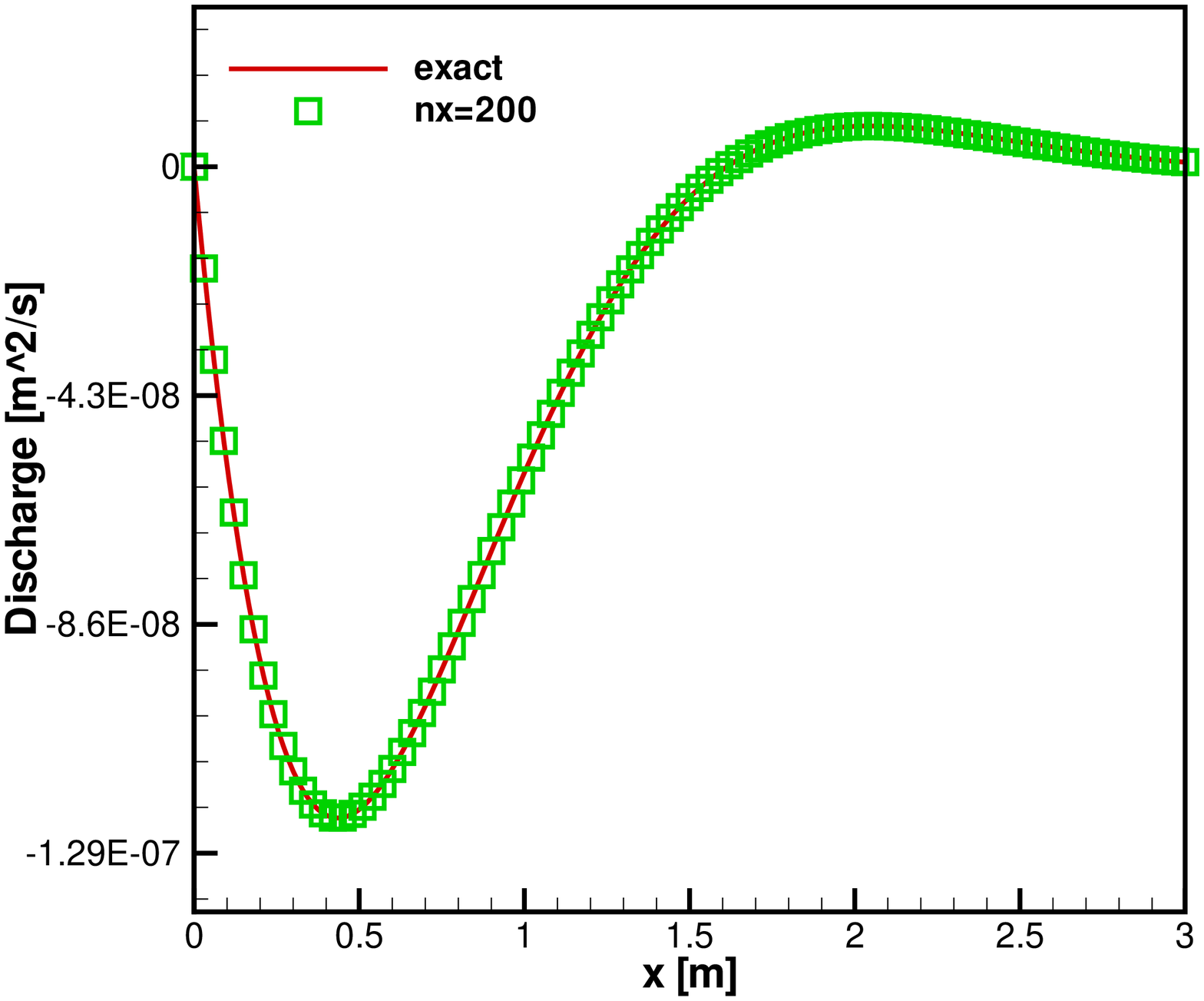}
\caption{ The numerical solutions of the propagation of a pulse to and from an expansion in Section \ref{wave-damping} with $ 200$ cells at $t = 25 \, s$. The damping of a discharge wave with $C_f = 0$  (upper left), $C_f = 0.000022$ (upper right), $C_f = 0.000202$ (lower left) and $C_f = 0.005053$ (lower right).  } \label{f:damping}
\end{figure}

\newpage

\section{Conclusions} \label{the-conclusion}

In this paper, we have presented a well-balanced finite difference WENO scheme to solve the blood flow model. A special splitting of the source term allows us to design specific approximations such that the resulting WENO scheme maintains the well-balanced property for steady state solutions,
and at the same time keeps theirs  original high order accuracy and essentially non-oscillatory
property for general solutions. Extensive numerical examples are given to demonstrate the well-balanced property, high order accuracy, and steep shock transitions of the proposed numerical scheme. The approach is
quite general and can be adapted to high order finite volume schemes and discontinuous Galerkin finite element
methods, which constitutes an ongoing work.

%Rigorous theoretical analysis as well as extensive numerical experiments all strongly suggest that the present
% schemes keep the steep shock transitions of the proposed numerical scheme.

\section*{Acknowledgements}
The research of the second author is supported
by the National Natural Science Foundation of P.R. China (No. 11201254, 11401332) and the Project for Scientific Plan of Higher Education in Shandong Providence of P.R. China (No. J12LI08). This work was partially performed at the State Key Laboratory of Science/Engineering Computing of P.R. China by virtue of the computational resources of Professor Li Yuan's group. The first author is also thankful to Professor Li Yuan for his kind invitation.

%The research of the second author is sponsored by NSF grant
%DMS-1216454, ORNL and the U. S. Department of Energy, Office of Advanced
%Scientific Computing Research. The work was partially performed at ORNL, which is managed
%by UT-Battelle, LLC, under Contract No. DE-AC05-00OR22725.
%The authors also sincerely acknowledge the constructive remarks of the referees that have led to many improvement in %this paper.
%The authors also appreciate the three anonymous
%reviewers especially the first one for the invaluable
%comments and suggestions, which play an important role to improve this paper.

% ¸½Â¼

%% References


%\bibliographystyle{elsarticle-num}



%\bigskip

%\bibliographystyle{plain}

\bibliographystyle{unsrt}
%\bibliography{refs}

\end{document}